\documentclass[11pt,reqno,fleqn]{amsart}

\usepackage{amsfonts,amsmath,amsthm,amssymb}
\usepackage{pdfsync}

\newtheorem{thm}{Theorem}
\newtheorem{defin}{Definition}
\newtheorem{lemma}{Lemma}
\newtheorem{prop}{Proposition}

\newtheorem{cor}{Corollary}

\newtheorem*{thmA}{Theorem A:}

\newtheorem*{prop1}{Proposition 1'}

\newcommand{\Z}{\ensuremath{\mathbb{Z}}}
\newcommand{\C}{\ensuremath{\mathbb{C}}}

\newcommand{\QQ}{\ensuremath{\mathbb{Q}}}
\newcommand{\F}{\ensuremath{\mathcal{F}}}

\newcommand{\1}{\ensuremath{\mathbf{1}}}
\newcommand{\x}{\ensuremath{\mathbf{x}}}
\newcommand{\Q}{\ensuremath{\mathbf{\mathcal{Q}}}}
\newcommand{\y}{\ensuremath{\mathbf{y}}}
\newcommand{\s}{\ensuremath{\mathbf{s}}}
\newcommand{\T}{\ensuremath{\mathbf{t}}}
\newcommand{\bb}{\ensuremath{\mathbf{b}}}
\newcommand{\z}{\ensuremath{\mathbf{z}}}
\newcommand{\uu}{\ensuremath{\mathbf{u}}}
\newcommand{\vv}{\ensuremath{\mathbf{v}}}
\newcommand{\w}{\ensuremath{\mathbf{w}}}
\newcommand{\h}{\ensuremath{\mathbf{h}}}

\newcommand{\B}{\ensuremath{\mathbf{\mathcal{B}}}}
\newcommand{\<}{\ensuremath{\lesssim}}

\newcommand{\subs}{\ensuremath{\subseteq}}

\newcommand{\W}{\ensuremath{\mathcal{W}}}
\newcommand{\lm}{\ensuremath{\preccurlyeq}}

\newcommand{\eps}{\ensuremath{\varepsilon}}
\newcommand{\la}{\ensuremath{\lambda}}
\newcommand{\La}{\ensuremath{\Lambda}}

\newcommand{\ga}{\ensuremath{\gamma}}

\newcommand{\ff}{\ensuremath{\mathfrak{F}}}
\newcommand{\f}{\ensuremath{\mathfrak{f}}}
\newcommand{\HH}{\ensuremath{\mathfrak{H}}}
\newcommand{\II}{\ensuremath{\mathfrak{I}}}

\newcommand{\G}{\ensuremath{\mathfrak{g}}}
\newcommand{\bG}{\ensuremath{\bar{\G}}}
\newcommand{\GG}{\ensuremath{\mathfrak{G}}}
\newcommand{\PPP}{\ensuremath{\mathfrak{P}}}
\newcommand{\0}{\ensuremath{\mathbf{0}}}

\newcommand{\eq}{\begin{equation}
\newcommand{\ee}{\end{equation}}}
\newcommand{\p}{\ensuremath{\mathfrak{p}}}

\numberwithin{equation}{section}

\begin{document}
\title{Diophantine equations in the primes}
\author{Brian Cook and \'Akos Magyar}
\thanks{
The second author is supported by NSERC grant 22R44824 and ERC-AdG. 321104.}

\begin{abstract}
 Let $\p=(\p_1,...,\p_r)$ be a system of $r$  polynomials with integer coefficients of degree $d$ in $n$ variables $\x=(x_1,...,x_n)$.  For a given  $r$-tuple of integers, say $\s$, a general local to global type statement is shown via classical Hardy-Littlewood type methods which provides sufficient conditions for the solubility of $\p(\x)=\s$ under the condition that each of the $x_i$'s is prime.\end{abstract}

\maketitle

%%%%%%%%%%%%%%%%%%%%%%%
\section{Background and main results}

Let $\p=(\p_1,...,\p_r)$ be a system of polynomials with integer coefficients of degree $d$ in the $n$ variables $\x=(x_1,...,x_n)$. Our primary concern is finding solutions with each coordinate prime, which we call prime solutions to the system of equations $\p(\x)=\s$, where $\s\in\Z^r$ is a fixed element. In more geometric terms, our task is to count prime points on the complex affine variety defined by this equation, which we will  denote by $V_{\p,\s}$.

If $\p$ is composed entirely of linear forms the known results may be split into a classical regime and a modern one. With $\p$ being a system of $r$ linear forms, define the rank of $\p$ to be the minimum number of nonzero coefficients in a non-trivial linear combination
\[\la_1\p_1+...+\la_{r}\p_{r},\]
and denote this quantity by $\B_1(\p)$. This quantity is of course positive if and only if the forms are linearly independent.  The classical results on the large scale distribution of prime points on  $V_{\p,\s}$ are conditional on the rank being sufficiently  large in terms of $r$ (for example, $2r+1$ follows from what is shown here). In this realm are  many well known results such as the ones due to Vinogradov \cite{Vin}, van der Corput, and more recently Balog \cite{Bal}. The modern results are mostly summed up in the work of Green and Tao \cite{GT1}, where the large scale distribution of prime points on $V_{\p,\s}$ is determined only on the condition that $\B_1(\p)$ is at least 3, a quantity independent of $r$. These results cover all scenarios that do not reduce to a binary problem. However, the most recent result is due to Zhang. He has shown, by extending the already stunning results of Goldston, Pintz, and Yildirim \cite{GPY}, that one of the equations\footnote{The number of needed equations has been reduced by a poly-math project.} $x_1-x_2=2i$, $i=1,...,35\times 10^6$ does have infinitely many prime solutions \cite{Zhang}.

The scenario for systems involving higher degree forms in certainly less clean cut \cite{Br}, and even the study of  integral points on $V_{\p,\s}$ is a non-trivial problem. General results for the large scale distribution of integral points are provided by  Birch \cite{Bi} and Schmidt \cite{Sch}, which again require the system to be large with respect to certain notions of rank (with respect to the number and degrees of the forms involved). Working within the limitations of these results, one should expect to be able to understand the large scale distribution of prime points as well. For systems of forms which are additive, for instance the single form $a_1x_1^d + ... + a_nx_n^d$, this is something that has been done and the primary result here is due to Hua \cite{Hua}. On the opposite end, if the system of forms is a  bilinear system, or even contains a large bilinear piece, one can also provide similar results, a particular instance of which is given by Liu \cite{Liu} for a quadratic form. However, there have previously been no results for general systems of nonlinear polynomials. Providing such a result is the aim of this work.

Before continuing it is worth discussing this problem from the point view of some recent results in sieve theory. Let $\Gamma\subset GL_n(\Z)$ be a finitely generated subgroup, $\mathcal{O}=\Gamma\mathbf{b}$ be the orbit for some $\mathbf{b}\in\Z^n$, and $g$ be an integral polynomial in $n$ variables. Bourgain, Gamburd, and Sarnak in \cite{BGS} initiated the study of a sieve method for finding points  $\y\in\mathcal{O}$ such that $g(\y)$ has few prime factors, provided of course there are no obvious constraints. In particular they prove for suitable groups that the almost prime points of the orbit form a Zariski dense subset, taking $g$ to be the monomial form $g(\y)=y_1\ldots y_n$.
In case if a form $\f$ is preserved by a sufficiently large group of transformations $\Gamma \subset GL_n(\Z)$, the level sets $V_{\f,\s}$ are partitioned into orbits of $\Gamma$ and the methods of \cite{BGS} show the existence of almost prime solutions. The advantage is that one does not need the largeness of the rank of the form, however, apart from quadratic forms, examples of such forms are quite rare (e.g. determinant forms), see  \cite{RS} for the study of integer solutions for such forms. Liu and Sarnak in \cite{LS} carry out this idea for indefinite quadratic forms in three variables, with $\Gamma$ being the subgroup of the special orthogonal group which preserves $\f$. In particular, they prove that one can find points on certain level sets of such quadratic forms where each coordinate has at most 30 prime factors.\\

Returning to the problem at hand, for a fixed system of polynomials $\p$ let us define for each prime $p$ the quantity
\[
\mu_p=\lim_{t\rightarrow \infty}\frac{(p^t)^r M(p^t)}{\phi^n(p^t)},
\]
provided the limit exists, where $M(p^t)$ represents the number of solutions to the equation $\p(\x)=\s$ in the multiplicative group of reduced residue classes $mod\,p^t$, denoted by $U_{p^t}^n$, and $\phi$ is Euler's totient function. A general heuristic argument suggests that we should have
\begin{eqnarray}\label{asy}
\mathcal{M}_{\p,\s}(N)&:=&\sum_{\x\in[N]^n} \La(\x) \mathbf{1}_{V_{\p,\s}}(\x)\nonumber\\
&\approx&\mu_\infty(N,\s)\prod_{p\,prime}\mu_p(\s)\,N^{n-D},
\end{eqnarray}
where $\mu_\infty(N,\s)$ is the singular integral that appears in the study of  integral points on $V_{\p,\s}$, see \cite{Bi},\cite{Sch}, $\La$ denotes the von Mangoldt function, $D=dr$,  and $\La(\x) =\La(x_1)...\La(x_n)$.

What is actually shown here is a precise result of this form for systems of polynomials of common degree provided that the system has large rank in the sense of Birch for the nonlinear forms and in the sense described above for linear forms. Let us be given a system of homogeneous polynomials $\f=(\f_1,...,\f_r)$  with integer coefficients of degree exactly $d$. Define the singular variety, over $\C^n$, associated to the forms  $\f$ to be the collection of $\x$ such that the Jacobian of $\f$ at $\x$, given by the matrix of partial derivatives
\[
Jac_{\,\f}(\x)=\left[\frac{\partial \f_{k}}{\partial x_j}(\x)\right]_{k=1,j=1}^{r,n},
\]
has rank strictly less than $r$. This collection is labeled as $V_{\f}^*$. The Birch rank $\B_d(\f)$ is defined  for $d>1$ and is given by  $codim(V_{\f}^*)$, provided that $r\neq0$. If $r=0$ then we simply assign the value $\infty$. This notion is extended to a general polynomial system $\p$ of degree $d$ by defining the rank by $\B_d(\p)=\B_d(\f)$, where $\f$ is the system of forms consisting of the highest degree homogeneous parts of the polynomials $\p$. In particular, if the a system of forms has positive Birch rank then the systems of forms $\f$ is a linearly independent system.

The main result that is shown here is the following.

%%%%%%%%%%%%%%%%%%%%%%%%%%%%%%%%%%%%%%%%%%%%%%%%%%%%%%%%%%%%%%
\newpage

\begin{thm}
For given positive integers $r$ and $d$, there exists a constant $\chi(r,d)$  such that the following holds:\\

Let $\p=\p^{(d)}$ be a given system of integral polynomials with $r$ polynomials of degree $d$ in $n$ variables, and set $D=dr$. If we have $\B_d(\p)\geq\chi(r,d)$  then for the equation $\p(\x)=\s$ we have an asymptotic of the form\[
\mathcal{M}_{\p,\s}(N)= \prod_p \mu_p(\s) \, \mu_\infty(N,\s) \, N^{n-D}+O(N^{n-D} (\log\,N)^{-1})\]
Moreover, if $\p(\x)=\s$ has a nonsingular solution in $\mathbb{U}_p$, the $p$-adic integer units, for all primes $p$, then \[ \prod_p \mu_p (\s)>0.\]
\end{thm}

Note that the singular integral $\mu_\infty(N,\s)$ is the same as in the work of Birch and Schmidt (see \cite{Sch}, Sec.9 and \cite{Bi}, Sec.6). It is positive provided that the variety $V_{\f,N^{-d}\s}$ has a nonsingular real point in the open cube $(0,1)^n$, and is bounded from below independently of $N$ if there is a nonsingular real point in the cube $(\eps,1-\eps)^n$ for some $\eps>0$. Indeed, the singular integral has the representation
\eq
\int_{V_{\f,\T}\cap[0,1]^n} d\sigma_{\f}(x),
\ee
where $\T=N^{-d}\,\s$, and $d\sigma_\f(x)$ is a positive measure on $V_{\f,\T}\backslash V_{\f}^*$ which is absolutely continuous with respect to the Euclidean surface area measure, see Birch \cite{Bi}, Sec.6.

The quantitative aspects of the constants $\chi(r,d)$ are in general extremely poor. The terms $\chi(r,1)$ may be taken to be $2r+1$. The case for quadratic forms is still somewhat reasonable,  for systems of quadratics one can show  $\chi(r,2)\leq2^{2^{Cr^2}}$ (to be compared to $r(r+1)$ for the integral analogue). However, the constants  $\chi(1,d)$ already exhibit tower type behavior in $d$ (to be compared to $d2^d$ for the integral analogue), and the situation further worsens from there.

A more detailed analysis of the singular series can be carried out it certain instances to provide more concrete statements. This is not done here, however there are some circumstances that have previously been considered. Borrowing the results in this direction from Hua and Vinogradov gives the following known results  (which are weaker than the best known for $d>1$) for diagonal forms as a corollary. Note that for a diagonal  form  $\f$ of degree $d$ in $n$ variables we have $\B_d(\f)=n$.

\begin{cor}
Theorem 1 applies to diagonal integral forms $\f(\x)=a_1x_1^d+a_2x_2^d+...+a_nx_n^d$,  $a_1a_2...a_n\neq0$, when $n\geq 3+(d-1)2^d$. In particular if $d=1$ and $a_1=a_2=a_3=1$ one recovers the well known fact that every sufficiently large odd integer  is a sum of three primes; if $d=2$ and $a_1=a_2=...=a_7=1$ then every sufficiently large integer congruent to 7 modulo 24 is a sum of seven squares of primes.
\end{cor}

 Liu provides a general scenario to guarantee that the singular series is positive for a single quadratic form. A quadratic form $\Q$ has the representation $\langle \x,A\x\rangle$ for some symmetric integral $n\times n$ matrix $A$. The Jacobian of the form is $2A\x$, giving that the singular variety $V^*_Q$ is the null space of $A$,  and so $\B(\Q)=rank(A)$. A general condition for $\Q$ to be well behaved with respect to each prime modulus is that $p^{n-2}$ does not divide $\det(A)$ for any prime $p$, as then a simple consideration of eigenvalues shows that the matrix $A$ has rank at least three when viewed over the finite field with $p$ elements.

\begin{cor}
Let $\Q(\x)=\langle \x,A\x\rangle$ be an indefinite integral quadratic form in $n$ variables with $rank(A)\geq\chi(2,1)$ and $p^{n-2}\nmid \det(A)$ for any prime $p$. Then $\Q(\x)=0$ has a solution in $\mathbb{U}_p$ for all $p$ if and only if $\Q(\x)=0$ has a solution with $x_i$ prime for each $i$.
\end{cor}

The value of $\chi(2,1)$ can be worked out in directly with what is done in this paper with minimal difficulty, and previously a value of 34 has been provided by the first author.  Applying the results of Liu, as mentioned in the following section, is however more efficient and  gives a value of 21.
For the related problem of finding integer solutions to a single quadratic equation $\Q(\x)=0$ in a set of positive upper density $A\subs\Z$,  Keil in \cite{Keil} has recently obtained the rank condition $rank(\Q)\geq 17$. The analogous statement for the integers is the famous Hasse-Minkowski Theorem which requires  no conditions on the rank or the determinant, and of course only looks at solubility in $\mathbb{Z}_p$ and not $\mathbb{U}_p$.

%%%%%%%%%%%%%%%%%

\bigskip
\section{Overview}

The primary technique used in the proof of Theorem 1 is the circle method, and the argument is an adaptation of the following mean value approach. If a single integral form $\ff$ of degree $d$ in $n$ variables takes the shape
\eq\label{2.0}
\ff(\x)=x_1^d+\ff_1(\y)+\ff_2(\z),
\ee
where $\x=(x_1,\y,\z)$, then we have the representation\[
\mathcal{M}_{\ff,0}=\int_{0}^1\left(\sum_{x_1\in [N]}\La(x_1)e(\alpha x_1^d)\right)\left(\sum_{\y\in [N]^m}\La(\y)e(\alpha \ff_1(\y))\right)\left(\sum_{\z\in [N]^{n-1-m}}\La(\z)e(\alpha \ff_2(\z))\right)d\alpha\]
\[
:=\int_{0}^1S_0(\alpha)S_1(\alpha)S_2(\alpha)d\alpha.\]
An application of the Cauchy-Schwarz inequality then gives
\[
\mathcal{M}_{\ff,0}^2\leq ||S_0||_\infty^2 ||S_1||_2^2||S_2||^2_2\leq ||S_0||^2_\infty(\log\,N)^{2n-2}Y(N)Z(N),
\]
where $Y(N)$ is the number of solutions to the equation $\ff_1(\y)=\ff_1(\y')$ with $\y,\y'\in[N]^m$ and
$Z(N)$ is the number of solutions to the equation $\ff_2(\z)=\ff_2(\z')$ with $\z,\z'\in[N]^{n-1-m}.$ If $\ff_1$ and $\ff_2$ are assumed to have large rank, then $Y(N)Z(N)=O(N^{2m-d})O(N^{2(n-1-m)-d})=O(N^{2n-2-2d}))$.  More generally, for any measurable subset $\mathfrak{u}\subset[0,1]$ we have
\eq\label{1}
\int_{\mathfrak{u}}S_0(\alpha)S_1(\alpha)S_2(\alpha)d\alpha=O(||S_0||_{\infty(\mathfrak{u})}(\log \,N)^{n-1}N^{n-1-d}),\ee
where $||S_0||_{\infty(\mathfrak{u})}$ denotes the supremum of $|S_0(\alpha)|$ for $\alpha\in\mathfrak{u}$.

Then a partition into so-called major arcs $\mathfrak{M}(C)$ and minor arcs $\mathfrak{m}(C)$, depending on a parameter $C$ (see Sec.5 for the precise definitions), becomes useful due to the following fundamental estimate of Hua and Vinogradov

\begin{lemma}\cite{Hua}
Given $c>0$, there exists a $C$ such that $||S_0||_{\infty(\mathfrak{m}(C))}\leq N(\log\,N)^{-c}.$
\end{lemma}
\noindent This, together with equation ~\ref{1}, in turn gives the bound
\eq
\int_{\mathfrak{m}(C)}S_0(\alpha)S_1(\alpha)S_2(\alpha)d\alpha=O((\log \,N)^{-1}N^{n-d}),
\ee
and one is left with the task of approximating the integral over the major arcs
\[
\int_{\mathfrak{M}(C)}\sum_{\x\in[N]^n}e(\alpha \ff(\x))d\alpha.\]
Without going into the details here, let us remark that the major arcs consist of very small intervals (or boxes) centered at rational points with small denominators, and the integral essentially depends only on the distribution of the values of the polynomials $\ff(x)$ in small residue classes. Then the approximation can be done via standard methods in this area, similarly as in the diagonal case \cite{Hua}.\\

Now let us look at the case of a general form $\ff$ of degree $2$. If we introduce a splitting of the variables $\x=(x_1,\y,\z)$, we induce a decomposition of the shape
\[
\ff(\x)=ax_1^2+\G^{(1)}(\y,\z)x_1^{}+\ff_1(\y)+\ff_2(\z)+\G^{(2)}(\y,\z)\]
for a form $\G^{(2)}$ which is bilinear in $\y$ and $\z$, and a linear form $\G^{(1)}$. There are two possible approaches to adapting the above argument to this case.

The first involves a dichotomized argument based on the rank of $\G^{(2)}$. If we have that $\G^{(2)}$ has large rank, one can obtain good bounds on the exponential sum
\[
\sum_{\y\in[N]^m}\sum_{\z\in[N]^{n-1-m}}\La(\y)\La(\z)e(\alpha(\ff_1(\y)+\ff_2(\z)+\G^{(2)}(\y,\z)+x_1\G^{(1)}(\y,\z))))
\]
by simply removing the contribution of the von Mangoldt function with two applications of the Cauchy-Schwarz inequality. During this process the $x_1$ variable differences out, and the $x_1$ summation may be treated trivially.  In this case the methods of Birch \cite{Bi} are applicable (and the rank bounds are comparable, see \cite{Liu}).  If $\G^{(2)}$ has small rank, then it must be the case that $\ff_1$ and $\ff_2$  each have large rank for appropriately chosen splitting of the variables. Write $\G^{(2)}(\y,\z)=\langle\y,B\z\rangle$ for an appropriately sized matrix $B$ whose rank is small, and split $\G^{(1)}(\y,\z)=\l_1(\y)+\l_2(\z)$. The above argument can then be run on the intersection of the level sets of $\l_1(\y),\,\l_2(\z)$, and $B\z$, as the form $\ff(\x)$ takes the diagonal shape given in \eqref{2.0}, and both $\ff_1$ and $\ff_2$ have large rank on this small co-dimensional affine linear space. On such an intersection we get an extra power gain, which is equal to its codimension which then compensates for the loss of originally applying the Cauchy-Schwarz inequality on each level set. Thus summing the estimates over all such level sets gives an appropriate bound.

The second approach, to be fair, is simply a streamlined version of the first which removes the need for  a dichotomized approach, and this is the one we shall follow.  The main requirement here is an appropriate decomposition of $\ff$ in the form
\[
\ff(\x)=ax_1^2+\G^{(1)}(\y,\z)x_1^{}+\ff_1(\y)+\ff_2(\z)+\G^{(2)}(\y,\z)
\]
such that the rank of $\ff_1+\G^{(2)}$ is sufficiently large, the number of variables of composing $\y$, say $m$, is controlled, and the rank of $\ff_2$ is large with respect to $m$. As before, we wish to fix $\l_1(\y),\,\l_2(\z)$, and $B\z$. The difference is that we have no assumption on the rank of $B$. However, by controlling the value of $m$, we have a way to control the number of linear equations in $\z$. Running the argument as before and summing over the level sets reduces the minor arcs estimates to providing an appropriate bound for the number of solutions to the system
\begin{eqnarray}\label{syst}
\ff_1(\y)+\G^{(2)}(\y,\z)&=&\ff_1(\y')+\G^{(2)}(\y',\z)    \nonumber \\
\l_1(\y)&=&\l_1(\y')    \, \, \, \,        \nonumber \\
\ff_2(\z)&=&\ff_2(\z')       \\
\l_2(\z)&=&\l_2(\z') \,\, \,      \nonumber \\
B\z&=&B\z'  \,\,\,\nonumber
\end{eqnarray}
with $\y,\y'\in[N]^{m}$ and $\z,\z'\in [N]^{n-1-m}$. This is achieved by the rank assumptions of $\ff_1+\G^{(2)}$ and  $\ff_2$ in the original decomposition.

The strategy for forms of higher degree starts by a similar decomposition of the form
\[
\ff(\x)=ax_1^d+\G^{(1)}(\y,\z)x_1^{d-1}+...+\G^{(d-1)}(\y,\z)x_1+\ff_1(\y)+\ff_2(\z)+\G^{(d)}(\y,\z),
\]
where the $\G^{(i)}$ are forms of degree $i$, and $\ff_1(\y),\,\ff_2(\z),\,\G^{(d)}$ are forms of degree $d$.
Again we require that the rank of $\ff_1+\G^{(d)}$ is large with respect to $\G^{(i)}$ for each $i<d$, the number of variables $m$ composing $\y$ is small, and  the rank of $\ff_2$ is large with respect to $m$. That such a decomposition is possible is the subject of Section 4. Then we view each form $\G^{(i)}$, $1\leq i \leq d$, as a sum of forms in $\y$ with coefficients that are forms in $\z$, and the number of these coefficients is bounded in terms of $m$ and $d$. On each of the level sets of this new system of forms in $\z$ we have a system of forms in $\y$, the number of which is bounded in terms of $d$. Now passing to the further level of sets of these forms $\y$ provides a place to carry out the simple Cauchy-Schwarz argument at the beginning of this section. Summing back over the level sets then provides a system analogous to the one above.

The only problem with this so far is that the system we end up with contains at least a portion of each form $\G^{(i)}(\y,\z)$ for $i=1,...,d$, and we have no control on the rank of these forms at all and therefore have no way of dealing with the terminal system. The solution to this problem is found in the work of Schmidt \cite{Sch}. His results provide a way of partitioning the level sets of a form by the  level sets of a system of forms that does have high rank in each degree. Section 3 is dedicated to this. Working with this more regular system as opposed to the $\G^{(i)}$'s does provide a manageable terminal system, and allows for a bound on the minor arc integral.

Extending this method to systems of forms is relatively straightforward at this point, and is of course carried out below. The major technical difficulty here is the need to isolate larger number of suitable variables $x_1,...,x_{K}$ to get the logarithmic gain on the minor arcs, as opposed to a randomly chosen single variable $x_1$.

\bigskip

\subsection{Outline and Notation}
The outline for the rest of the paper is as follows. Sections 3 and 4 are as described above. The completion of the bound for the  integral over the minor arcs is going to be carried out in Section 5. The major arcs are dealt with in Section 6, where the asymptotic formula is shown. Section 7 is dedicated to the proof of Theorem 1. The final section concludes the work with a few further remarks.

\bigskip
\noindent\textbf{Remarks on notation}
The symbols $\mathbb{Z}$, $\mathbb{Q}$, $\mathbb{R}$, and $\mathbb{C}$ denote the integers, the rational numbers, the real numbers, and the complex numbers, respectively. The $r$-dimensional flat torus $\mathbb{R}^r\slash \mathbb{Z}^r$ is denoted by $\mathbb{T}^r$. The $p$-adic integers are denoted by $\mathbb{Z}_p$, and the units of $\Z_p$ are denoted by $\mathbb{U}_p$. The symbol $Z_N$ represents shorthand for the groups  $\mathbb{Z}\slash N\Z$. Also, the shorthand for the multiplicative group reduced residue classes $Z_N^*$ is $U_N$.

For a given measurable set $X\subseteq\mathbb{T}^r$ we shall use the notation $||f||_{p(X)}$ to denote the $L^p$ norm of the function $\mathbf{1}_Xf$  with the normalized Lebesgue measure on the $r$-dimensional flat torus. If $X$ is omitted it is assumed that $X=\mathbb{T}^r$. Here, and in general,  $\mathbf{1}_X$ denotes a characteristic function for  $X$ in a specified ambient space, and, on occasion, the set $X$ is replaced by a conditional statement which defines it.

The Landau $o$ and $O$ notation is used throughout the work. The notation $f\< g$ is sometimes used to replace $f=O(g)$. The implied constants are independent of $N$ and $\s$, but may depend on all other parameters such as $d$, $r$, $n$, and $\p$.

\bigskip
\section{A regularity lemma}
In \cite{Sch}, Schmidt provides an alternative definition of rank for a form. For a single form $\ff$ of degree at least 2 defined over a field $k$, define the \textit{Schmidt rank} $h_k(\ff)$ to be the minimum value of $l$ such that there exists a decomposition
\[
\ff=\sum_{i=1}^lU_iV_i,
\]
where $U_i$ and $V_i$ are forms defined over $k$ of degree at least one. For a system $\f^{(d)}=(\f_{1}^{(d)},...,\f_{r_d}^{(d)})$ of forms of degree $d$ we define  $h_k(\f)$ to be
\[
\min\{h_k(\lambda_1\f^{(d)}_1+...+\lambda_{r_d}\f^{(d)}_{r_d})\,:\, \lambda_i\neq0 \, for \, some \, i\}.\]

The following basic properties of the Schmidt rank will be used later.\\

\begin{itemize}
\item If $\f$ is defined over a field $k$, and $k'$ is an extension of $k$, then $h_{k'}(\f)\leq h_k(\f)$\\

\item The Schmidt rank is invariant under invertible linear transformations of $k$, i.e. $h_k(\f\circ A)=h_k(\f)$ for $A\in GL_n(k)$.\\

\item If $\f'(x_2,...,x_n)=\f(0,x_2,...,x_n)$, then $h_k(\ff')\geq h_k(\ff)-1$. \\
\end{itemize}
The first two are clear from the definition, and the third simply follows from the fact that $\f(\x)-\f'(\x)$ is of the form $x_1\G(\x)$ for some $d-1$ degree system of forms $\G$. Also, the second and third imply that the rank cannot drop on a codimension $j$ subspace of $k^n$ by more than $j$.

As observed by Schmidt (see \cite{Sch}, Lemma 16.1), the Birch rank $\B_d(\ff)$ and the complex Schmidt rank $h_\C(\ff)$ are essentially equivalent for a form of degree $d$, each being bounded  by a constant times the  other. For example, if $\Q(\x)=\langle \x,A\x\rangle$ is a quadratic form, then $\B(\Q)$ is $rank(A)$ (as pointed out in the opening section) and $h_\mathbb{C}(\Q)=l$ is the smallest integer greater than or equal to $rank(A)/2$ (which follows from the fact that $\Q$ is  equivalent over $\mathbb{C}$ to the form $x_1^2+...+x_l^2$ and $x^2_1+x_2^2$ factors over $\mathbb{C}$) .   Of course the same phenomenon is true for systems as well. The rational Schmidt rank $h_\mathbb{Q}$ on the other hand is not equivalent to the Birch rank and we need the following result which is a weakened version of a central result in \cite{Sch}.
A definition is required.

\begin{defin}
Let $\p=(\p^{(d)},...,\p^{(1)})$ be a graded system of polynomials with rational coefficients, meaning $\p^{(i)}$ is the subfamily of $\p$ of polynomials of degree precisely $i$.
Assume that $\p^{(i)}$ consists of $r_i$ polynomials for each $1\leq i \leq d$, and set  $D=\sum_i ir_i$.\\ The system $\p$ is said to be regular if $|V_{\p,\mathbf{0}}\cap [N]^n|=O(N^{n-D})$, as $N\to\infty$.
\end{defin}

\begin{thmA}[Schmidt \cite{Sch}]
For a given positive integers $R$ and $d$, there exists constants $\rho_i(R,d)$ for $2\leq i\leq d$ such that the following holds:\\

Let $\p=(\p^{(d)},...,\p^{(2)})$ be a graded system of rational polynomials and let $\f^{(i)}$ be the system of forms consisting of the homogeneous parts of the polynomials $\p^{(i)}$.
Let $r_i=|\f^{(i)}|$, $R=r_2+...+r_d$ the total number of forms, and $D=2r_2+...+dr_d$ the total degree of the system.
If we have $h_\mathbb{Q}(\f^{(i)})\geq\rho_i(R,d)$ for each $i$, then the system $\p$ is regular.
\end{thmA}

Indeed this follows immediately for homogeneous systems, that is when $\p=\f$, from Theorem II, or alternatively from Propositions I, II and III in \cite{Sch}. For non-homogeneous systems all propositions and hence Theorem II continues to hold, see Sec.9 and the last paragraph of Sec.10 in \cite{Sch}. The constants $\rho_i(r,d)$ are given explicitly satisfying the bound $\rho_i(r,d)\leq R r_i\, 2^{Cd\,log\,d}$.

One of the key observations of Schmidt is that his definition of rank has a very nice reductive quality with respect to the degree, in the sense that  forms of small rank may be replaced by a small number of forms of lesser degree. The next result captures this idea.

\begin{prop}[Regularization of systems]
Let $d>1$ be a fixed integer, and let $F$ be any collection of non-decreasing functions $F_i(R)$ for $i=2,...,d$ mapping the nonnegative integers into themselves. For a collection of non-negative integers $r_1,...,r_d$, there exist constants \[C_1(r_1,...,r_d,F),...,C_d(r_1,...,r_d,F)\] such that the following holds: \\ \hfill

Given a system of integral forms $\f=(\f^{(d)},\f^{(d-1)},..., \f^{(1)})$, where each of the $\f^{(i)}$ is a system of $r_i$ forms of degree $i$, there exists a system of rational forms $\G=(\G^{(d)},\G^{(d-1)},..., \G^{(1)})$ satisfying:
\begin{enumerate}
\item Each form of the system $\f$ can be written as a rational polynomial expression of the forms of the system $\G$. In particular, the level sets of $\G$ partition those of $\f$.
\item The number of forms in each subsystem of  $\G^{(i)}$, say $r'_i$, is at most $C_i(r_1,...,r_d,F)$ for each $1\leq i\leq d$.
\item The system $((\G^{(d)},\G^{(d-1)},..., \G^{(2)}))$ satisfies the rank condition $h_\mathbb{Q}(\G^{(i)})\geq F_i(R')$ for each $2\leq i \leq d$, with $R'=r_1'+\ldots +r_d'$, moreover the system $\G^{(1)}$ is a linearly independent family of linear forms.

%\item Given a partition of the variables $\x=(\y,\z)$ there exists a system of forms $\G$ satisfying conditions (1),(2),(3) such that the following also holds. For any $1\leq k\leq d$ let $\G^{(k)}_1(\y,\z),\ldots,\G^{(k)}_s(\y,\z)$ denote all forms of the system $\G^{(k)}(\y,\z)$ which depend on $\y$ variables, so are not forms of the $\z$ variables alone. Then for any $(\la_1,\ldots,\la_s)\neq\0$ one has that the form
   % \[h(\y,\z)=\la_1 \G^{(k)}_1(\y,\z)+\ldots \la_s \G^{(k)}(\y,\z)\] also depend on $\y$ variables, so is not a form of the $\z$ variables alone.

%Set $H$ to be the linear subspace defined by $\G^{(1)}=0$, and $R'$ to be the total number of forms in $\G$ of degree at least 2. The system $((\G^{(d)},\G^{(d-1)},..., \G^{(2)}))$ has $h_\mathbb{Q}(\G^{(i)}|_H)\geq F_i(R')$ for each $2\leq i \leq d$.  Here $\G^{(i)}|_H$ denotes the restriction of $\G^{(i)}$ to the subspace $H$.

\end{enumerate}
\end{prop}

If a system $\G$ satisfying properties (1) and (3) then we call it a \emph{regularization} of the system $\f$ with respect to the family of functions $F$.

It is worth noting that results of this type have been previously obtained over fields with positive characteristic. See \cite{GT3} for such a result over fields with high characteristic, i.e.  larger than $d$ where $d$ is as stated in the proposition. See \cite{KLo} for a result in the low characteristic case.

\begin{proof}
The proof is carried out by a double induction on the parameters. First for a fixed  $d$ we show that the case $r_d$ with any choice of $r_{d-1},...,r_1$ implies the similar scenario for the case $r_d+1$. Then the induction on $d$ is carried out.

The initial case we need to consider is $d=2$ with a given function $F_2(R)$. Take a system of forms $\f=(\f^{(2)},\f^{(1)})$ with $r_2=1$ and any  value of $r_1$.  If  $h_\mathbb{Q}(\f^{(2)})\geq F_2(R)$, with $R=r_1+1$, then we may simply take $\G=\f$, otherwise $\f^{(2)}=\sum_{i=1}^lU_iV_i$ for some rational linear forms $U_i$ and $V_i$ where  $l<F_2(R)$. We may then adjoin the linear forms $U_1,...,V_l$ to the system $\f^{(1)}$ to obtain the system $\G^{(1)}$, and let $\G$ be a maximal linearly independent subsystem of $\G^{(1)}$. Properties (1) and (2) are easily verified for this system, and property (3) is immediate.

Now for a fixed value of $d$ assume that the result holds for all systems with maximal degree $d$ for any given collection of  functions $F$ when $r_d=j$ and $r_{d-1},...,r_1$ are arbitrary. Consider now a fixed collection of functions $F$ and a system $\f=(\f^{(d)},...,\f^{(1)})$ with $r_d=j+1$. Let $\f '$ be the system $(\f^{(d-1)},...,\f^{(1)})$. By the induction hypothesis, there is a system
$\G '$ of rational forms which is a regularization of $\f'$ with respect to $F'_i(R):=F_i(R+(j+1))$ for $i=2,...,(d-1)$.

Now let $\tilde{\G}'=(\f^{(d)},\G ')$. If $\tilde{\G}'$ fails to be the regularization of $\f$ with respect to the family of functions $F$, then  $h_\mathbb{Q}(\f^{(d)})< F_d(R_{\tilde{\G}'}+(j+1))$, where $R_{\tilde{\G}'}$ is the number of forms of the system $\tilde{\G}'$. As before, in this case there must exist homogeneous rational polynomials $U_i$ and $V_i$, $i=1,...,l<F_d(R_{\tilde{\G}'}+(j+1))$, such that
\[
\lambda_1\f^{(d)}_1+...\lambda_{j+1} \f^{(d)}_{j+1}=\sum_{i\leq l}U_iV_i,
\]
where without loss of generality we may assume that $\lambda_{j+1}\neq0$. Now let $\G''$ be $\G'$ adjoined with the those forms $U_i$ and $V_i$ which are not linear combinations of forms already in $\G'$, and set \[
\tilde{\G}''=((\f^{(d)}_1,...,\f^{(d)}_{j}),\G'').\]
By the induction hypothesis there is a system $\G$ which is the regularization of $\tilde{\G}''$ with respect to initial collection of functions  $F$. It is clear from the construction of the system $\G$ that all forms of the initial system are polynomial expressions of the ones in $\G$, and the number of forms in each subsystem $\G^{(i)}$ is expressible in terms of $r_1,...,r_d$, and $d$. Thus the system  $\G$ is the regularization of $\f$ satisfying conditions $(1)$, $(2)$ and $(3)$.
The induction argument to go  from $d$ to $d+1$ is simply the above argument carried out  with $j=0$.
\end{proof}

%Now given a partition of the variables $\x=(\y,\z)$ assume that condition (4) is not fulfilled for the system $\G^{(k)}(\y,\z)$ for a given value of $1\leq k\leq d$. Then there exists an $s$-tuple of rational numbers $(\la_1,\ldots,\la_s)$ such that, say $\la_s\neq 0$ and
%\[\la_1\G^{(k)}_1(\y,\z)+\ldots+\la_s \G^{(k)}_s(\y,\z)=h^{(k)}(\z).\]
%Let $\tilde{\G}^{(k)}$ be the system of forms where $\G^{(k)}_s(\y,\z)$ is replaced by the form $h^{(k)}(\z)$, and let $\tilde{\G}^{(l)}={\G}^{(l)}$ for $1\leq l\leq d$, $l\neq k$. It is clear that the rational pencils of the systems $\tilde{\G}^{(k)}$ and ${\G}^{(k)}$ are the same and hence the system $\tilde{\G}$ also satisfies properties (1), (2) and (3).On the other hand the system $\tilde{\G}$ has a smaller number of forms which depend on some of the $\y$ variables than the system $\G$, thus the system with the minimum number of such forms satisfying properties (1),(2) and (3) must satisfy property (4) as well.

We will need a stronger version of the above proposition relative to a partition of the variables $\x=(\y,\z)$. To state it let us introduce a modified version of the Schmidt rank. For a single form $\ff(\y,\z)$ of degree at least 2 defined over a field $k$, we define its \textit{Schmidt rank} with respect to the variables $\z$, $h_k(\ff;\z)$ to be the minimum value of $l$ such that there exists a decomposition
\[
\ff(\y,\z)=\sum_{i=1}^lU_i(\y,\z)V_i(\y,\z)+W(\z),
\]
where $U_i$ and $V_i$ are forms defined over $k$ of degree at least one. For a system $\f^{(d)}=(\f_{1}^{(d)},...,\f_{r_d}^{(d)})$ of forms of degree $d$ we define  $h_k(\f;\z)$ to be
\[
\min\{h_k(\lambda_1\f^{(d)}_1+...+\lambda_{r_d}\f^{(d)}_{r_d};\z)\,:\, \lambda_i\neq0 \, for \, some \, i\},\]
and set $h_k(\f;\z)=\infty$ if $\f=\emptyset$.
Note that $h_k(\f;\z)\leq h_k(\f)$ and $h_k(\f;\z)=0$ if and only if a non-trivial linear combination of the forms in $\f$ depends only on the variables $\z$.

\begin{prop1}[Regularization of systems, parametric version]
Let $d>1$ be a fixed integer, and let $F$ be any collection of non-decreasing functions $F_i(R)$ for $i=2,...,d$ mapping the nonnegative integers into themselves. For a collection of non-negative integers $r_1,...,r_d$, there exist constants \[C_1(r_1,...,r_d,F),...,C_d(r_1,...,r_d,F)\] such that the following holds: \\ \hfill

Given a system of integral forms $\f=(\f^{(d)},\f^{(d-1)},..., \f^{(1)})$ where each of the $\f^{(i)}$ is a system of $r_i$ forms of degree $i$ and a partition of the variables $\x=(\y,\z)$, there exists a system of rational forms $\G=(\G^{(d)},\G^{(d-1)},..., \G^{(1)})$ satisfying:
\begin{enumerate}
\item Each form of the system $\f$ can be written as a rational polynomial expression of the forms of the system $\G$. In particular, the level sets of $\G$ partition those of $\f$.
\item The number of forms in each subsystem of  $\G^{(i)}$, say $r'_i$, is at most $C_i(r_1,...,r_d,F)$ for each $1\leq i\leq d$.
\item The system $(\G^{(d)},\G^{(d-1)},..., \G^{(2)})$ satisfies the rank condition $h_\mathbb{Q}(\G^{(i)})\geq F_i(R')$ for each $2\leq i \leq d$, with $R'=r_1'+\ldots +r_d'$, moreover the system $\G^{(1)}$ is a linearly independent family of linear forms.
\item  The system $(\bar{\G}^{(d)},\bar{\G}^{(d-1)},..., \bar{\G}^{(2)})$ satisfies the modified rank condition $h_\mathbb{Q}(\bar{\G}^{(i)};\z)\geq F_i(R')$ for each $2\leq i \leq d$, where the subsystem $\bar{\G}^{(i)}$ is obtained from the system $\G^{(i)}$ by removing the forms depending only on the $\z$ variables.
\end{enumerate}
\end{prop1}

We call a system $\G$ satisfying the conclusions of Proposition 1' to be a strong regularization of the system $\f$ with respect to the variables $\z$.

\begin{proof} The argument is a slight modification of the proof of Proposition 1. To a given system $\f=(\f^{(d)},\f^{(d-1)},..., \f^{(1)})$ we assign its index which is the triple $(d,r_d,r_d')$, where $r_d=|\f^{(d)}|$ resp. $\bar{r}_d=|\bar{\f}^{(d)}|$ are the number of degree $d$ forms resp. the number of degree $d$ forms depending on at least one of the $\y$ variables. We will proceed via induction on the lexicographic ordering of the indexes. To be precise we deem $(d,r_d,\bar{r}_d)\prec (e,r_e,\bar{r}_e)$ if $d<e$, $d=e$ but $r_d<r_e$, or $d=e$, $r_d=r_e$ but $\bar{r}_d<\bar{r}_e$.

If $d=1$ all one needs to do is to choose maximal linearly independent subsystem of $\f=\f^{(1)}$. So assume $d\geq 2$ (and $r_d\geq 1)$, $I=(d,r_d,\bar{r}_d)$ is a fixed index and the result holds for any system with index $I'\prec I$. Again, let $\f'=(\f^{(d-1)},..., \f^{(1)})$ and let $\G'$ be a strong regularization of $\f'$ with respect to the functions $F_i'(R)=F_i(R+r_d)$ $(i=2,\ldots,d-1)$ and the variables $\z$. Let $\tilde{\G}'=(\f^{(d)},\G')$.
If $\tilde{\G}'$ fails to be a strong regularization of $\f$ with respect to the family of functions $F$ and the variables $\z$, then there are two possible cases.

Either $h_\mathbb{Q}(\f^{(d)})< F_d(R_{\tilde{\G}'}+r_d)$, where $R_{\tilde{\G}'}$ is the number of forms of the system $\tilde{\G}'$. As before, in this case one can replace the system $\tilde{\G}'$ with with a system $\tilde{\G}''$ of index $I''\prec I$. Then by the induction hypotheses there is a system $\G$ which is a strong regularization of $\tilde{\G}''$ with respect to initial collection of functions  $F$ and the variables $\z$. The system $\G$ will satisfy the claims of the proposition.

Otherwise $\bar{r}_d\geq 1$ and we have $h_\mathbb{Q}(\bar{\f}^{(d)};\z)< F_d(R_{\tilde{\G}'}+r_d)$, hence there must exist homogeneous rational polynomials $U_i$ and $V_i$, $i=1,...,l<F_d(R_{\tilde{\G}'}+r_d)$, and a form $Q(\z)$ of degree $d$ such that
\[
\lambda_1\f^{(d)}_1+...\lambda_{\bar{r}_d} \f^{(d)}_{\bar{r}_d}=\sum_{i\leq l}U_i(\y,\z)V_i(\y,\z)+Q(\z),\quad\quad (\bar{\f}^{(d)}=(\f^{(d)}_1,\ldots,\f^{(d)}_{\bar{r}_d}))
\]
where without loss of generality we may assume that $\lambda_{\bar{r}_d}\neq 0$. Now let $\G''$ be $\G'$ adjoined with the those forms $U_i$ and $V_i$ which are not linear combinations of forms already in $\G'$, and set \[
\tilde{\G}''=((\f^{(d)}\backslash \f^{(d)}_{\bar{r}_d},Q),\G'').\]
Note that we have replaced the form $\f^{(d)}_{\bar{r}_d}(\y,\z)$ with the form $Q(\z)$ hence the index of $\tilde{\G}''$ is the triple $I''=(d,r_d,\bar{r}_d-1)\prec I$. Again,
there is a system $\G$ which is a strong regularization of $\tilde{\G}''$ with respect to initial collection of functions  $F$ and the variables $\z$. As explained before the system $\G$ satisfies the claims of the proposition. Note that the procedure depends on the partition $\x=(\y,\z)$, however since there only finitely many ways to partition the variables $\x$, the constants $C_i(r_1,\ldots,r_d,F)$ can be taken independent of the partition.
\end{proof}

Applying Proposition 1' with the functions being given by the values of the Schmidt constants $\rho_i(R,d)$ then provides the following.

\begin{cor}
Let $\f=(\f^{(d)},...,\f^{(2)})$ be given system of rational  forms with $r_i$ forms of degree $i$ composing each subsystem $\f^{(i)}$ and let $\x=(\y,\z)$ be a partition of the independent variables. There exists a regular system of forms $\G$ satisfying the conclusions of Proposition 1'.
\end{cor}

\begin{proof} For fixed $d$, let $\G$ be a strong regularization of the system $\f$ with respect to the functions $F_i(R):=\rho_i(R,d)+R$ and the variables $z$. We show that $|V_{\G,\mathbf{0}}\cap [N]^n|=O(N^{n-R})$.

Let $H:=V_{\G^{(1)},\mathbf{0}}$ be the nullset of the system $\G^{(1)}$, and let $A\in GL_n(\mathbb{Q})$ be a linear transformation such that $M:=A(H)$ is coordinate subspace of codimension $r_1$. It is easy to see that $A([N]^n)\subs K_1^{-1}\cdot [-K_2N,K_2N]^n$ for some constants $K_1$ and $K_2$, where by $\la\cdot\x$ we mean a dilation of the point $\x$ by a factor $\la$. Let $\G'=\G\circ A^{-1}$, then $V_{\G',\mathbf{0}}=A(V_{\G,\mathbf{0}})$, thus by homogeneity we have that
\eq\label{2.*}
|V_{\G,\mathbf{0}}\cap [N]^n|\leq |V_{\G',\mathbf{0}}\cap [-KN,KN]^n|,\ee
with $K=K_1K_2$. We have that $h_{\mathbb{Q}}(\G'^{(i)})=h_{\mathbb{Q}}(\G^{(i)})\geq F_i(R)$, and hence $h_{\mathbb{Q}}(\G'^{(i)}|_M)\geq \rho_i(R,d)$, for $i=2,\ldots,d$, where $R$ denotes the total number forms of $\G'$.
The quantity on the right side of \eqref{2.*} is the number of integral points $\x$ in $M\cap [-KN,KN]^n = [-KN,KN]^{n-r_1}$ such that $\G'^{(i)}(\x)=0$ for $i=2,\ldots,d$. Since the system $(\G'^{(d)},\ldots,\G'^{(2)})$ satisfies the conditions of Theorem A, we have that $|V_{\G',\mathbf{0}}\cap [-KN,KN]^n|=O(N^{n-R})$. This proves the corollary.
\end{proof}

The somewhat technical last conclusion in Proposition 1' is utilized in the following lemma which will play an important role in our minor arcs estimates.

\begin{lemma} Let $\G^{(k)}(\y,\z)=(\G^{(k)}_1(\y,\z),\ldots,\G^{(k)}_s(\y,\z))$ be system of homogeneous forms of degree $k$. Then for the system $\tilde{\G}^{(k)}(\y,\y',\z)=(\G^{(k)}(\y,\z),\G^{(k)}(\y',\z))$ we have that
\[h_\mathbb{Q}(\tilde{\G}^{(k)};\z)=h_\mathbb{Q}(\G^{(k)};\z).\] Here $\y$, $\y'$ and $\z$ represent distinct sets of variables.
\end{lemma}

\begin{proof} Since $\G^{(k)}$ is a subsystem of $\tilde{\G}^{(k)}$ it is clear that $h_\mathbb{Q}(\tilde{\G}^{(k)};\z)\leq h_\mathbb{Q}(\G^{(k)};\z)$. Now let $\la$, $\mu$ be $s$-tuples of rational numbers not all 0. Using the short hand notation $\la\cdot \G^{(k)}=\sum_i \la_i \G^{(k)}_i$  assume that we have a decomposition
\eq\label{double}\la\cdot \G^{(k)}(\y,\z)+\mu\cdot \G^{(k)}(\y',\z)=\sum_{i=1}^h U_i(\y,\y',\z)V_i(\y,\y',\z)+Q(\z),\ee
where all forms $U_i$, $V_i$ have positive degree. We argue that $h\geq h_\mathbb{Q}(\G^{(k)})$. Assuming without loss of generality that $\la\neq\0$ and substituting $\y'=\0$ into \eqref{double} we have
\[\la\cdot \G^{(k)}(\y,\z) = \sum_{i=1}^h U_i(\y,\0,\z)V_i(\y,\0,\z)+Q'(\z),\] with $Q'(\z)=Q(\z)-\mu\cdot \G^{(k)}(\0,\z)$. For each $1\leq i\leq h$ the term $U_i(\y,\0,\z)V_i(\y,\0,\z)$ vanishes identically or both forms $U_i(\y,\0,\z)$ and $V_i(\y,\0,\z)$ have positive degrees, thus by the definition of the modified Schmidt rank we have that $h\geq h_\mathbb{Q}(\G^{(k)};\z)$.
\end{proof}

\bigskip

\section{A decomposition of forms}

For $I\subset [n]$, let
$\y=(y_i)_{i\in I}$  be the vector with components $y_i=x_i$ for $i\in I$. Also let $\z$ be the vector defined similarly  for the set $[n]\backslash I$.  Note that $(\y,\z)=\x$. To such a partition of the variables there corresponds a unique decomposition of a form $\f$
\[\f(\x)=\f_1(\y)+\G(\y,\z)+\f_2(\z),\]
where $\f_1(\y)$ is the sum of all monomials in $\f$ which depends only on $\y$ variables, $\G(\y,\z)$ is the sum of the monomials depending both on $\y$ and $\z$, and $\f_2(\z)$ is the sum of the remaining terms. The decomposition extends to a system $\f$ in the obvious manner.
The aim of this section is to prove the following result.

\begin{prop}\label{decomp}
Let positive integers $C_{1}$ and  $C_{2}$ be given.  Let $\f$ be given system of $r$ rational  forms  with $\B_d(\f)$ sufficiently large with respect to $C_{1}$, $C_{2}$, $r$ and $d\geq1$. There exists an $I \subset [n]$ such that $|I|\leq C_{1}r$ and the associated decomposition
\[
\f(\x)=\f_1(\y)+\f_2(\z)+\G(\y,\z)
\]
satisfies $\B_d(\f_1+\G)\geq C_{1}$ and $\B_d(\f_2)\geq C_{2}$.
\end{prop}

The proof of this result is carried out at the end of this section, and is done so by dealing directly with the Jacobian matrices. Some notation is helpful. Let $M=M(\x)$ be a $i\times j$ matrix whose entries depend on $\x$. The notation $M\lm M'$ is used to imply that $M$ is a submatrix of $M'$ obtained by the deletion of columns, so that $M$ is an $i\times j'$ matrix with $j'\leq j$. Let $V^*_M$ be the collection of $\x$ where $M=M(\x)$ has rank strictly less than $i$.
Clearly one has that if $M\lm M'$, then $V^*_{M'}\subseteq V^*_M$.

\begin{lemma}
If $\f$ is a system of $r$ integral forms of degree $d>1$ in $n$ variables, then the restriction of $\f$ to the hyperplane defined by $x_n=0$ has rank at least $\B_d(\f)-r$.  In the case $d=1$ we have the improved lower bound $\B_1(\f)-1.$
\end{lemma}

\begin{proof}
First we consider the case $d>1.$ Denote the restriction of $\f$ to the subspace defined by $x_n=0$ as $\F$. The matrix $Jac_\F$ is then the matrix $Jac_\f$ with the last column deleted and restricted to the space $x_n=0$. It follows that $V_\F^*\cap H\subseteq V_\f^*\cap\{x_n=0\}$, where $H$ denotes the variety  where the last column of $Jac_\f$ has all entries equal to zero. As $H$ is defined by at most $r$ equations,  it has co-dimension  at most $r$ (\cite{Har}, Ch. 7).  In turn it follows from the sub-additivity of the dimension of intersection of varieties that $\B_d(\F)+r\geq \B_d(\f)$.

Now let $\f=(\f_1,...,\f_r)$ be a system of $r$ linear forms, and $\F=(\F_1,...,\F_r)$ be the system of forms obtained by setting $x_n=0$. We have for any $\la_1,...,\la_r$\[
\la_1\f(\x)+...+\la_r\f_r(\x)=\la_1\F(\x)+...+\la_r\F_r(\x)+(\la_1a_{1,n}+...+\la_ra_{r,n})x_n,\]
where $a_{i,n}$ is the $x_n$ coefficient in $\f_i(\x)$. From this the result is clear from the definition of $\B_1$.
\end{proof}

Note that the Birch rank of a system of forms $\F$ is invariant under invertible linear transformations via the multivariate chain rule, and this fact remains true for a non-homogeneous system, as its rank is defined to be the rank of its homogeneous part. Thus we have more generally

\begin{cor}\label{subspacerank}
If $H$ is an affine linear space of co-dimension $m$, then the restriction of $\f$ to $H$ has rank at least $\B_d(\f)-mr$ when $d>1$. In the case $d=1$ we have the improved lower bound $\B_1(\f)-m.$
\end{cor}

Now define $\mathcal{C}_\f(k)$ to be the minimal value of $m$ such that there exists a matrix valued function $M\lm Jac_\f$ of size $r\times m$ such that $V^*_M$ has dimension at most $n-k$.  This is defined to be  infinite if no such value exists. Thus $\mathcal{C}_\f(k)\leq l$ if one can select $l$ columns of the Jacobian matrix $Jac_\f$ forming a matrix $M$ so that $V^*_M$ has co-dimension at least $k$.

\begin{lemma}\label{3}
For a system $\f$ of $r$ forms of degree $d>1$, one has that $\mathcal{C}_\f(k)\leq kr$ as long as\\  $(k-1)((d-1)r)^{k-1}< \B(\f)$.
\end{lemma}

\begin{proof}

 Write the singular variety $V_\f^*$ as an intersection of varieties $V_I$, where $V_I$ is a the zero set of the determinant of the $r\times r$ minor coming from the selecting the columns $I=\{i_1,\ldots,i_r\}\subset [n]$. Let $\II_{n,r}$ denote the family of $r$-element subsets of $[n]$.

Proceeding inductively, we show that one can select index sets $I_1,\ldots,I_l$, $I_j\in\II_{n,r}$, so that $codim (\bigcap_{j=1}^l V_{I_j})=l$, as long as $\B_d(\f)$ is sufficiently large with respect to $l,d$ and $r$. Indeed, assume that this is true for a given $l$ and let $V^{(l)}:=\bigcap_{j=1}^l V_{I_{j}}$. The degree of each hypersurface $V_I$ is at most $r(d-1)$, and hence by the basic properties of the degree, see \cite{Har}, Ch.7, $V^{(l)}$ has at most $((d-1)r)^l$ irreducible components. Label the components with dimension precisely $n-l$ as $Y_1,...,Y_J$, where $J\leq ((d-1)r)^l$.

For each $Y_j$, set $N(j)$ to be the set of $I\subs\II_{n,r}$ such that  $Y_j\subs V_I$. If there is an index set  $I$ such that $I\notin\bigcup_{j=1}^J N(j)$, then $\dim(V_I\cap V^{(l)})=n-l-1$, and one can choose $I_{l+1}=I$.
Otherwise it follows that $\cup_{j=1}^J N(j)=\II_{n,r}$. In this case, we have
$V_\f^*=\cap_{j=1}^J(\cap_{I\in N(j)}V_I)\,$,  and in turn
\[codim(V_\f^*)\leq\sum_{j=1}^J\,codim(\bigcap_{I\in N(j)}V_I)\leq \sum_{j=1}^J\,codim(Y_j)\leq l((d-1)r)^l,
\]
which cannot happen if $l((d-1)r)^l<\B_d(\f)$. Thus one can choose a $I_{l+1}$ such that $V^{(l+1)}$ has co-dimension $l+1$. Now let $I':=\cup_{j=1}^l I_j$ and let $M=M_{I'}$ be the associated submatrix of $Jac_\f$ obtained by selecting the columns belonging to $I'$. It is clear that $V_M^*\subs \cap_{j=1}^l V_{I_j}$ and hence $codim(V_M^*)\geq l$ while $|I'|\leq lr$.
\end{proof}

\begin{proof}[Proof of Proposition 2]

The proof is carried out separately for the cases $d=1$ and $d>1$, starting with the latter. We have $\B_d(\f)=\B_d(\f^{(d)})$ so we may assume that $\f$ consists of forms of degree $d$. Start by applying Lemma ~\ref{3} with $k=C_1$, valid by assuming that $\B_d(f)$ is sufficiently large. Then there are at most $C_1r$ columns of $Jac_{\f}$ providing a sub-matrix $M\lm Jac_\f$ with $codim\,(V_M^*)\geq C_1$. Let $I$ denote the collection of the indices of these columns, noting that $m:=|I|\leq C_1r$. Let $\y=(x_i)_{i\in I}$ and $\z=(x_i)_{i\notin I}$ and $\f(\x)=\f(\y,\z)=\f_1(\y)+\G(\y,\z)+\f_2(\z)$ be the associated decomposition of the system $\f$. It is easily seen that $M\lm Jac_{\f_1+\G}$, and it follows that $V_{\f_1+\G}^*\subseteq V^*_M$, and hence $\B(\f_1+\G)\geq C_1$.

Now look at the the matrix $W=W(\y,\z)$ obtained by deleting the columns of $M$ from $Jac_\f$. One now has $Jac_{\f_2}(\z)=W(\0,\z)$ as $\partial\f/\partial\z\,(\0,\z)=\partial\f_2/\partial\z\,(\z)$. Let $H$ denote the variety of points $\z$ so that $M(\0,\z)=\mathbf{0}$, then we have that $\,V_{\f_2^*}\cap H\subs \{\z;\ (\0,\z)\in V_\f^*\}\,$ hence
\[codim\,(V_{\f_2}^*) +C_1r^2\geq \B_d(\f)-C_1r,\]
as $codim\,(H)\leq C_1r^2$ and $m\leq C_1r$. Thus assuming $\B_d(\f)\geq C_2+C_1(r^2+r)$ provides the bound $\B_d(\f_2)\geq C_2$ as required.

When $d=1$ write the linear system $\f$ as $A\x$ for some integral matrix $A$ of size $r\times n$. Note that the term $\G$ does not appear in this case, and  we are simply looking for a block decomposition  $A=[A_1|A_2]$, where  each $A_i$ is an $r\times n_i$ matrix with $n_1+n_2=n$ . If the linear system has $\B_1(\f)$ at least $C_1r+C_2$ then there is an $r\times r$ submatrix $B_1$ which has full rank.  Now write $A=[B_1|A']$ . By  Corollary ~\ref{subspacerank} we have that the linear system defined by $A'$ in the variables $(x_{r+1},...,x_n)$ has rank at least $(C_1-1)r+C_2$. Repeating this process gives a block decomposition $A=[B_1|B_2|...|B_{C_1}|A_2]$ where each $B_i$ is $r\times r$ full rank matrix. We now set $A_1:=[B_1|...|B_{C_r}]$. As each $B_i$ has full rank, it follows the rank of the linear system defined by $A_1$ in the variables $x_1,..,x_{rC_1}$ has rank at least $C_1$ because any non-trivial linear combination of the associated forms must contain at least one non-zero coefficient for one variable from each of the sets $\{x_{jr+1},...,x_{(j+1)r}\}$ for $0\leq j\leq C_1-1$. Applying Corollary ~\ref{subspacerank}  shows that the system defined by $A_2$ in the remaining variables has rank at least $C_2$.
\end{proof}

\bigskip

\section{The minor arcs}

Assume now  throughout this section that we have a fixed  system of integral polynomials $\p=(\p_1,...,\p_r)$, where each $\p_i$ is of degree $d$.  The system $\f$ is again  the highest degree homogenous parts of $\p$.

For a given value of $C>0$ and an integer $q\leq (\log\,N)^C$, define a major arc \[\mathfrak{M}_{\mathbf{a},q}(C)=\{\alpha\in[0,1]^r:\max_{1\leq i\leq r}|\alpha_i-a_i/q|\leq N^{-d}(\log\,N)^C\}\] for each $\mathbf{a}=(a_1,...,a_r)\in U_q^r$. When $q=1$ it is to be understood that $U_{1}=\{0\}$. These arcs are disjoint, and the union
\[
\bigcup_{q\leq (\log\,N)^C}\bigcup_{\mathbf{a}\in U_q^r}\mathfrak{M}_{\mathbf{a},q}(C)
\]
defines the major arcs $\mathfrak{M}(C)$. The minor arcs are then given by
\[\mathfrak{m}(C)=[0,1]^r\backslash \mathfrak{M}(C).\]

The main result in this section is to deal with the integral representation on the minor arcs.

\begin{prop}\label{minor}
There exists constant $\chi(r,d)$ such that if we have $\B_d(\p)\geq \chi(r,d)$, then the following holds. To any given value $c>0$ there exists a $C>0$ such that
\eq
\int_{\mathfrak{m}(C)}e(-\s\cdot\alpha) \sum_{\x\in[N]^n}\La(\x)e(\p(\x)\cdot \alpha)\,d\alpha=O(N^{n-D}(\log\,N)^{-c}),
\ee
with an implied constant independent of $\s$.
\end{prop}

Another set of minor arcs is also required for an exponential sum estimate.  For each $1\leq i\leq d$, define for $a^{(i)}\in U_q$  the major arc\[
\mathfrak{N}^{(i)}_{a^{(i)},q}(C)=\{\xi^{(i)}\in\mathbb{T}:\, |\xi^{(i)}-\frac{a^{(i)}}{q}|\leq N^{-i}(\log\,N)^C\}.\]
Set \[
\mathfrak{N}_{\mathbf{a},q}(C)=\mathfrak{N}^{(d)}_{a^{(d)},q}(C)\times...\times\mathfrak{N}^{(1)}_{a^{(1)},q}(C),\]
where $\mathbf{a}=a^{(d)}\times ... \times a^{(1)}$. The major arcs are now
 \[\mathfrak{N}(C)=\bigcup_{q\leq (\log \, N)^C}\bigcup_{\mathbf{a}\in U_q^d} \mathfrak{N}_{\mathbf{a},q}(C).\]
Let $\mathfrak{n}(C)$ denote a set of minor arcs $\mathfrak{n}(C)=[0,1]^d\backslash \mathfrak{N}(C)$.

Define the exponential sum \[
S_0(\beta)=\sum_{x\in [N]}\La(x)e(\beta_dx^d+...+\beta_1x)\]
for $(\beta_d,...,\beta_1)\in\mathbb{T}^d$.

\begin{lemma}\label{hua} (Hua)
Given $c>0$, there exists a $C$ such that $||S_0||_{\infty(\mathfrak{n}(C))}=O (N(\log\,N)^{-c}).$
\end{lemma}

\noindent For a proof the reader is referred to  \cite{Hua} (Ch. 10,  \S5, Lemma 10.8).

\begin{proof}[Proof of Proposition 3]

Our first goal is to pick an appropriate splitting of the variables  of the form $\x=(x_1,...,x_K,\y,\z)$ which then  induces a decomposition of the form\[
\p(\x)=\p_0(x_1,...,x_K,\y,\z)+\p_1(\y)+\G(\y,\z)+ \p_2(\z)\]
such that three things hold: First we need $x_1,...,x_K$ selected which are useful for applying Lemma \ref{hua}. Secondly we need  $\y$ consisting of $m$ variables  chosen so that $\p_1+\G$ has large rank with respect to $K$. Lastly we need the rank of $\p_2$ to be large with respect to $K$ and $m$. The last two conditions are going to be achieved  by an application of Proposition 2  assuming that the rank of $\p$ is sufficiently large.

We select the variables $x_1,...,x_K$ first; we  consider the associated system of forms $\f$ of degree $d$. We  collect the $r$ coefficients of  each term  $x_{i_1}...x_{i_d}$ into a vector $\bb_{i_1,...,i_d}$.  We select $r$ of these which are linearly independent, this is possible as the system $\f$ is linearly independent. The total number of indices involved is our value of $K$, in this choice is at most $dr$, and we assume that  the corresponding variables to be the first $1\leq i \leq K$. The variables $x_1,...,x_K$ have now been selected.

For any choice of $\y$ and $\z$, i.e. any generic splitting of the variables $(x_{K+1},...,x_n)$, we have a decomposition of the shape
\begin{eqnarray}
\p(x_1,...,x_K,\y,\z)&=&\p(x_1,...,x_K,0,...,0)+\sum_{j=1}^{d-1}\sum_{1\leq i_1<....<i_j\leq K}\left(\sum_{k=1}^{d-j}\GG_{i_1,...,i_j}^{(k)}(\y,\z)\right)x_{i_1}...x_{i_j} \nonumber\\
& & +\sum_{\kappa=1}^{d-1}\sum_{1\leq \iota_1<....<\iota_\kappa\leq m}\left(\sum_{k=1}^{d-\kappa}\HH_{0;\iota_1,...,\iota_\kappa}^{(k)}(\z)\right)y_{\iota_1}...y_{\iota_{\kappa}}   \nonumber \\
 & & + \p_1(0,...,0,\y,0)+\p_2(0,...,0,0,\z),
\end{eqnarray}
where for each appropriate set of indices the $\GG_{i_1,...,i_j}^{(k)}$ and the $\HH_{0;\iota_1,...,\iota_\kappa}^{(k)}$ are systems of at most $r$ integral forms of degree $k$ in the appropriate variables. We write $\p_1(\y)$ for $\p_1(0,...,0,\y,0)$, and similarly $\p_2(\z)$ for $\p_2(0,...,0,0,\z)$.

Let $\GG$ be the collection of  all the forms $\GG_{i_1,...,i_j}^{(k)}$. These are independent of the choice of the $(\y,\z)$ partition and there are crudely at most $R_\GG\leq d^{2}K^{d}r\leq d^{d+2}r^{d+1}$ of them.  By Proposition 1', there is a system $\G_\GG$ which is a strong regularization of $\GG$ with respect to the functions  $F_i(R)=\rho_i(2R+2r_d,d)$ ($i=1,...,d-1$) and the variables $\z$ so that the number of forms of degree $i$ in $\G_\GG$ is bounded in terms of $r$ and $d$ independently of the choice of the partition of the variables. Correspondingly for any choice of the variables $(\y,\z)$ we have $R_{\G_\GG}\leq R_0$ for an appropriate constant $R_0= R_0(r,d)$.

The variables $\y$, and thus $\z$ as well,  are now chosen by Proposition 2 with the choice of $C_1=2^{d-1}\rho_d(2R_0+2r_d,d)$ so that the forms \[
\f(0,...,0,\y,\z)=\f_1(\y)+\G(\y,\z)+\f_2(\z)\]
have $\B_d(\f_1(\y)+\G(\y,\z))\geq C_1$ with the number of $\y$ variables, $m$, being at most $C_1 r$. With this choice we have that the system obtained by adjoining either of the systems $\f_1+\G$ or $\p_1+\p_3$ to the $\G_\GG$  is a regular system by Theorem A, where $\p_3$ is defined to be $\p(0,...,0,\y,\z)-\p_1(\y)-\p_2(z)$. Moreover  $h_\mathbb{Q}(\bar{\G}_\GG^{(i)};\z)\geq \rho_i(2R_{\G_\GG}+2r_d,d)$ for $i=1,...,d-1$ by Proposition 1', with $\bar{\G}_\GG$ denoting the subsystem obtained by removing the forms of $\G_\GG$ depending only on the $\z$ variables.

We break down further the forms of  $\G_\GG(\y,\z)$ by separating the $\y$ and $\z$ parts:
\eq(\G_{\GG})_i^{(l)}(\y,\z)=\sum_{\kappa=0}^{l}\sum_{1\leq \iota_1<....<\iota_\kappa\leq m}\HH^{(k;l)}_{i;\iota_1,...,\iota_\kappa}(\z)y_{\iota_1}...y_{\iota_{\kappa}} .\ee
Note that the right hand side introduces at most $lm^l\leq dm^d$ forms in $\z$ for the $i$th form of degree $l$ in $\G_\GG$. We collect the forms $\HH_{0;\iota_1,...,\iota_\kappa}^{(k)}$ and $\HH^{(k;l)}_{i;\iota_1,...,\iota_\kappa}$ into a system $\HH$.
Then the number of forms $R_\HH$ of $\HH$ is at most $R_{\G_\GG}dm^d+rd^2m^d$. Now we regularize the system $\HH$ with respect to the functions $F_i(R)=\rho_i(2R+2r,d)$ for $1\leq i \leq d-1$ and call the resulting system $\G_\HH$. Note that $|\G_\HH|\leq R_1$ where $R_1$ is a constant depending only on the original parameters $R$ and $d$.

If the system $\f_2(\z)$ has rank at least as large as $C_2=\rho_d(2R_1+2r_d,d)$, then system $\f_2$ adjoined to the system $\G_\HH$ is a regular system by Theorem A. This can be guaranteed as long the rank of $\f(0,...,0,\y,\z)$ is sufficiently large with respect to our choices of $C_1$ and $C_2$ by appealing to Proposition 2. As losing the first $K$ variables can drop the rank by at most $K(r+1)\leq r(r+1)d$, and
$C_1$ and $C_2$ are dependent only on $d$ and $r$, this is our choice of $\chi(r,d)$.

We now define the following sets:
\[
W_z(H)=\{\z\in[N]^{n-K-m}:\, \G_\HH(\z)=H  \}  ,\]
\[
W_y(G;H)=\{\y\in[N]^m:\,  \bar{\G}_\GG(\y,W_z(H))=G  \}.\]
Note that for $\z\in W_z(H)$ the system $\G_\GG(\y,\z)$ are of the form $q(\y,H)$ as all of its coefficient forms are constants depending only on the parameter $H$. The number of $H$  required is $N^{D_{\G_\HH}}$. The image of $[N]^{n-K}$ under $\bar{\G}_\GG$ is $O(N^{D_{\bar{\G}_\GG}})$, and this is an upper bound of the number of $G$'s for any fixed $H$, where the implied constant does not depend on $H$.

For any choice of $\z\in W_z(H)$ and $\y\in W_y(G;H)$, the polynomials $\p$ now take the shape given in (2.1)
\begin{eqnarray}
\p(x_1,...,x_K,\y,\z)&=&\p(x_1,...,x_K,0,...,0)+\sum_{j=1}^{d-1}\sum_{1\leq i_1<....<i_j\leq K}\mathbf{c}^{(d)}_{i_1,...,i_j}(G,H)x_{i_1}...x_{i_j} \nonumber\\
& & +\sum_{\kappa=1}^{d-1}\sum_{1\leq \iota_1<....<\iota_\kappa\leq m}\mathbf{c}^{(d)}_{0;i_1,...,i_j}(H)y_{\iota_1}...y_{\iota_{\kappa}}   \nonumber \\
 & & + \p_1(0,...,0,\y,0)+\p_2(0,...,0,0,\z) \nonumber \\
 &:= &\PPP_{0}(x_1,...,x_K,G,H) +\PPP_{1}(\y,H) +\p_{2}(\z).
\end{eqnarray}

Indeed on the level set $W_z(H)$ the values of the forms $\G_\HH$ and in turn those of the forms $\HH$ are fixed, including the forms of $\G_\GG$ depending only on the $\z$ variables. Then for $\y\in W_y(G;H)$ the values of the forms $\bar{\G}_\GG$ and hence the values of the all forms in $\G_\GG$ are fixed, which in turn fix the values of coefficient forms $\GG$. Define the exponential sums
\[
S_0(\alpha,G,H)=\sum_{x_1,...,x_K\in[N]} \La(x_1)...\La(x_K)e(\alpha\cdot \PPP_{0}(x_1,...,x_k,G,H)).\]
\[
S_1(\alpha,G,H)=\sum_{\y\in W_y(G;H)} \La(\y)e(\alpha\cdot \PPP_{1}(\y,H)).\]
\[
S_2(\alpha,H)=\sum_{\z\in W_z(H)} \La(\z)e(\alpha\cdot \p_{2}(\z)).\]
Now we have to bound the expression
\[
E_C(N)=\sum_H\sum_G\int_{\mathfrak{m}(C)}S_0(\alpha,G,H)S_1(\alpha,G,H)S_2(\alpha,H)e(-\s\cdot\alpha)d\alpha.\]

Proceeding as in Section 2.1, we obtain \eq
(E_C(N))^2\lesssim \left((\log\,N)^{2n}N^{D_{\bG_\GG}+D_{\G_\HH}}\sup_{H,G}||S_0(\cdot,G,H)||^2_{\infty(\mathfrak{m}(C))}\right)\sum_H\sum_G ||S_1(\cdot,G,H)||_2^2||S_2(\cdot,H)||^2_2 \ee
The summands on the right hand side can be expressed as the number of solutions of $\PPP_{1}(\y,H)=\PPP_{1}(\y',H)$ for $\y,\y'\in W_y(G;H)$ times the number of solutions to $\p_{2}(\z)=\p_{2}(\z')$ for $\z,\z'\in W_z(H)$. The conditions $\z,\z'\in W_z(H)$ may be replaced by the conditions $\z,\z'\in[N]^{(n-K-m)}$ and $\G_\HH(\z)=\G_\HH(\z')=H$. The conditions $\y,\y'\in W_2(G;H)$ may be replaced by the conditions $\y,\y'\in[N]^{m}$ and $\bar{\G}_\GG(\y,\z)=\bar{\G}_\GG(\y',\z)=G$ for any $\z\in W_z(H)$.

In short, we are summing over all $G$ and $H$ the number of solutions to the system
\begin{eqnarray}
\PPP_{1}(\y,H)&=&\PPP_{1}(\y',H) \nonumber\\
\bG_\GG(\y,\z)&=&\bG_\GG(\y',\z)=G \nonumber\\
\p_{2}(\z)&=&\p_{2}(\z') \nonumber\\
\G_\HH(\z)&=&\G_\HH(\z')=H \nonumber
\end{eqnarray}
for $\y,\y'\in[N]^{m}$ and $\z,\z'\in[N]^{(n-K-m)}$. With a little rearrangement this becomes
\begin{eqnarray}
\PPP_{1}(\y,\G_\HH(\z))&=&\PPP_{1}(\y',\G_\HH(\z) )\nonumber\\
\bG_\GG(\y,\G_\HH(\z))&=&\bG_\GG(\y',\G_\HH(\z))=G \nonumber\\
\p_{2}(\z)&=&\p_{2}(\z') \nonumber\\
\G_\HH(\z)&=&\G_\HH(\z')=H, \nonumber
\end{eqnarray}
and summing over $G$ and $H$ now simply removes two of the equalities. And after doing so, by removing  $\G_\HH$ as an argument puts us in the final form
\begin{eqnarray}
\p_{1}(\y)+\p_3(\y,\z)&=&\p_{1}(\y')+\p_3(\y',\z)\nonumber\\
\bG_\GG(\y,\z)&=&\bG_\GG(\y',\z) \nonumber\\
\p_{2}(\z)&=&\p_{2}(\z') \\
\G_\HH(\z)&=&\G_\HH(\z'), \nonumber
\end{eqnarray}
for $\y,\y'\in[N]^{m}$ and $\z,\z'\in[N]^{(n-K-m)}$.

Let us call the number of solutions to the system (5.6) $\mathcal{W}$. Then (5.5) takes the form  \eq
(E_C(N))^2\leq O\left((\log\,N)^{2n}\mathcal{W}N^{D_{\bG_\GG}+D_{\G_\HH}}\sup_{H,G}||S_0(\cdot,G,H)||^2_{\infty(\mathfrak{m}(C))}\right).
\ee
Proposition 3 follows immediately from the following two claims.

\bigskip
\noindent\textbf{Claim 1:}  \textit{Given $c>0$, there is  a $C$ such that the bound} \eq||S_0(\cdot,G,H)||_{\infty(\mathfrak{m}(C))}=O((\log\,N)^{-c}N^{K})\ee \textit{holds uniformly in $G$ and $H$}.
\bigskip

\noindent\textbf{Claim 2:} \textit{With the rank of } $\p_{1}(\y)+\p_3(\y,\z)$ \textit{and} $\p_{2}(\z)$ \textit{sufficiently large, the bound} \eq\label{claim2}\mathcal{W}=O(N^{2(n-K)-2D_\p -D_{\G_\HH}-D_{\bG_\GG}})\ee
\textit{holds.}
\bigskip

Let us start with claim 1.  We look at $\alpha\cdot\PPP_{0}(x_1,...,x_K,G,H)$, focusing on the coefficients of terms of the form $x_{i_1}...x_{i_d}$ for $1\leq i_j \leq K$. From our choice of $x_1,...,x_K$, there is a collection of indices, say $(i^{\kappa}_1,...,i^{\kappa}_d)$ for each $1\leq \kappa\leq r$, such that  the collection $\{\bb_{i^{\kappa}_1,...,i^{\kappa}_{d}}\}$  is linearly independent. Let $M$ denote the $r\times r$ matrix of these coefficient vectors as rows. The coefficient of $x_{i^{\kappa}_1}...x_{i^{\kappa}_d}$ in $\alpha^{}\cdot \PPP^{}_{0}(x_1,...,x_K,G,H)$ is $(M\alpha^{})_\kappa$.  Because $M$ has full rank, there is some term of the form $x_{i^{\kappa}_1}...x_{i^{\kappa}_d}$ with a coefficient $\beta$ where $\beta \in \mathfrak{m}^{}(C')$ for some slightly larger $C'$.

If it happens to be the case that the indices $i^{\kappa}_1,...,i^{\kappa}_d$ are equal, say all 1, then the bound follows directly from the bound in Lemma 5 for the $x_1$ summation, and claim 1 follows by treating the other sums trivially. Otherwise we assume that  $x_{i^{\kappa}_1}...x_{i^{\kappa}_d}=x_1^{\gamma_1}...,x_l^{\gamma_l}$ where $\sum\gamma_i=d$ and $l<d$. Now look at the sum $S_0$ in the form \[
\sum_{x_{l+1},...,x_K\in[N]}\La(x_{l+1})...\La(x_K)\sum_{x_1,...,x_l\in[N]}\La(x_1)...\La(x_l)e(\beta x_1^{\gamma_1}...x_l^{\gamma_l}+Q(x_1,...,x_K,G,H))\]
where $Q(x_1,...,x_K,G,H)$ is viewed as a polynomial in $x_1,...,x_l$ of degree less than $d$ with coefficients in the other $x_i$ and the $G$ and the $H$. We proceed by using the Weyl differencing method, that is applying the Cauchy-Schwarz inequality $\gamma_i$ times to the inner sum $T_0$ for each of the variables $x_i$.  This gives the upper bound \begin{eqnarray}
|T_0|^{2^d}\leq (\log\,N)^{d2^d}N^{2^dl-d-l}\sum_{x_1,...,x_l}\sum_{w^1_1,...,w^1_{\gamma_1}}  ...\sum_{w^l_1...w^1_{\gamma_l}}\left(\prod_{i=1}^l\Delta_{w^{i}_1,...,w^i_{\gamma_l}}\1_{x_i\in[N]}\right)\,e(\beta w_1^1...w_{\gamma_l}^l),\nonumber\end{eqnarray}
where $\Delta_wf(x)=f(x+w)f(x)$ is the multiplicative differencing operator, and $\Delta_{w_1,w_2}=\Delta_{w_2}(\Delta_{w_1})$, and so on. For a fixed value of $x$-variables, the inner sum is taken over a convex set $\mathcal{K}(x)\subs [-N,N]^d$. Renaming the $w$-variables as $\w=(w_1,\ldots,w_d)$ and fixing $w_1,\ldots,w_{d-1}$, the sum in $w_d$ is taken over an interval of length at most $2N$ of a phase which is linear in $w_d$ and is estimated by $\min\,(2N,\frac{1}{\|\beta w_1\ldots w_{d-1}\|})$. Here $\|\ga\|$ denotes the distance of a real number $\ga$ to the nearest integer.
Thus we have that
\[|T_0|^{2^d}\leq (\log\,N)^{d2^d}N^{2^dl-d}\sum_{(w_1,...,w_{d-1})\in [-N,N]^{d-1}}\ \min\,\left(2N,\frac{1}{\|\beta w_1\ldots w_{d-1}\|}\right).\nonumber\]
At this point we proceed as in \cite{Bi} Sec.2, see Lemmas 2.1-2.4, however as our context is slightly different we briefly describe the argument below. First note that by the multi-linearity of the expression in the denominator, we have that the sum in the $w$-variables is bounded by $2N\log N\,$-times the cardinality of the set
\[\mathcal{A}_N:=\left\{\w=(w_1,\ldots w_{d-1})\in [-N,N]^{d-1};\ \|\beta w_1\ldots w_{d-1}\|\leq \frac{1}{N}\right\}.\]
For given $1\leq M<N$ define the sets
\[\mathcal{A}_{N,M}:=\left\{\w=(w_1,\ldots w_{d-1})\in \left[-\frac{N}{M},\frac{N}{M}\right]^{d-1};\ \|\beta w_1\ldots w_{d-1}\|\leq \frac{1}{N M^{d-1}}\right\}.\]
Applying Lemma 2.3 in \cite{Bi} - a result from the geometry of numbers - successively in the variables $w_i$, we have that
\[|\mathcal{A}_N| \< M^{d-1}\ |\mathcal{A}_{N,M}|.\]
Choose $M:=N (\log\,N)^{-C''}$ with $C''=C'/d$, and notice if there is a vector $\w=(w_1,\ldots,w_{d-1})\in \mathcal{A}_{N,M}$ so that $q:=|w_1\ldots w_{d-1}|\geq 1$, then $\|q\beta\|\leq N^{-d}\,(\log\,N)^{C'}$ and hence $\beta$ would be in the major arcs $\mathcal{M}(C')$ contradicting our assumption. Thus all vectors $\w\in \mathcal{A}_{N,M}$ have at least one zero coordinate which implies that $|\mathcal{A}_{N,M}|\< (\log\,N)^{(d-2)C''}$, which gives the upper bound
\[|T_0|^{2^d} \< N^{2^d l}\ (\log\,N)^{d2^d+1-C''} \leq  N^{2^d l}\ (\log\,N)^{-C''/2},\]
if $C$, and hence $C'$ and $C''$, is chosen sufficiently large with respect to $d$. Claim 1 follows by taking the $2^d$th root and summing trivially in $x_{l+1},...,x_K$.

\bigskip

\bigskip

We turn to the proof of claim 2. One may write $\mathcal{W}=\sum_z T_1(\z) T_2(\z)$ where, for fixed $\z$, $T_1(\z)$ is the number of solutions $\y,\y'\in [N]^m$ to the system
\begin{eqnarray}
\p_{1}(\y)+\p_3(\y,\z)&=&\p_{1}(\y')+\p_3(\y',\z)\nonumber\\
\bG_\GG(\y,\z)&=&\bG_\GG(\y',\z), \nonumber
\end{eqnarray}
while $T_2(\z)$ is the number of solutions $\z'\in [N]^{n-K-m}$ to
\begin{eqnarray}
\p_{2}(\z)&=&\p_{2}(\z') \nonumber \\
\G_\HH(\z)&=&\G_\HH(\z'), \nonumber
\end{eqnarray}
By the Cauchy-Schwarz inequality we have that $\mathcal{W}^2\leq
(\sum_z T_1(\z)^2)\,(\sum_z T_2(\z)^2)=:\mathcal{W}_1\mathcal{W}_2$.

Now, $\mathcal{W}_1$ is the number of $\y,\y',\uu,\uu'\in [N]^m$ and $\z\in [N]^{n-K-m}$ satisfying the equations
\begin{eqnarray} \label{sys1}
\p_{1}(\y)+\p_3(\y,\z)-\p_{1}(\y')-\p_3(\y',\z)&=&0\nonumber\\
\p_{1}(\uu)+\p_3(\uu,\z)-\p_{1}(\uu')-\p_3(\uu',\z)&=&0\nonumber\\
\bG_\GG(\y,\z)-\bG_\GG(\y',\z)&=&0, \nonumber\\
\bG_\GG(\uu,\z)-\bG_\GG(\uu',\z)&=&0,
\end{eqnarray}
and $\W_2$ is the number of $\z,\z',\vv'\in [N]^{n-K-m}$ satisfying
\eq\label{sys2}
\p_{2}(\z)-\p_{2}(\z')=0,\quad \p_{2}(\z)-\p_{2}(\vv')=0\\
\G_\HH(\z)-\G_\HH(\z')=0,\quad \G_\HH(\z)-\G_\HH(\vv')=0
\ee

First we consider the system in \eqref{sys1}, and estimate the rank of the family of degree $d$ forms. For given $\la,\mu\in\mathbb{Q}^r$ assume we have the decomposition
\[\la\cdot (\f_{1}(\y)+\G(\y,\z)-\f_{1}(\y')-\G(\y',\z))+\mu\cdot (\f_{1}(\uu)+\G(\uu,\z)-\f_{1}(\uu')-\G(\uu',\z))=\sum_{i=1}^h U_i V_i,\]
where $U_i(\y,\y',\uu,\uu',\z)$ and $V_i(\y,\y',\uu,\uu',\z)$ are homogeneous forms of positive degree. Using the facts that $\f_1(\0)=\0$ and $\G(\0,\z)$ vanishes identically, substituting $\y'=\uu=\uu'=\0$ gives
\[\la\cdot (\f_{1}(\y)+\G(\y,\z))=\sum_{i=1}^h U_i'(\y,\z) V_i'(\y,\z),\]
thus $h\geq h_\mathbb{Q}(\f_1+\G)\geq 2^{1-d}\B_d(\f_1+\G)\geq \rho_d(2R_{\bG_\GG}+2r_d)$. Here we used the fact that $h_\mathbb{Q}(\f)\geq h_\mathbb{C}(\f)\geq 2^{1-d} \B_d(\f)$ for a homogeneous form $\f$ of degree $d$, see \cite{Sch}, Lemma 16.1. To estimate the rank of the degree $i$ forms for given $2\leq i\leq d-1$ we invoke Lemma 2. We have that
\[
h_\QQ((\bG^{(i)}_\GG(\y,\z)-\bG^{(i)}_\GG(\y',\z),\bG^{(i)}_\GG(\uu,\z)-\bG^{(i)}_\GG(\uu',\z));\z)=
h_\QQ(\bG^{(i)}_\GG(\y,\z)-\bG^{(i)}_\GG(\y',\z);\z)\geq\]
\[\geq h_\QQ(\bG^{(i)}_\GG(\y,\z),\bG^{(i)}_\GG(\y',\z);\z)=h_\QQ(\bG^{(i)}_\GG(\y,\z);\z)\geq \rho_i(2R_{\bG_\GG}+2r_d)\]\\
Since $h_\mathbb{Q}(f(\y,\z))\geq h_\QQ(f(\y,\z);\z)$ for any system of forms $\f$ and the total number of forms of system \eqref{sys1} is $2R_{\bG_\GG}+2r_d$, one has by Theorem A
\[\W_1= O( N^{4m+n-K-m-2D_\p-2D_{\bar{\G}_\GG}}).\]

For system \eqref{sys2} the estimates are simpler. Similarly as above it is easy to see directly from the definition that forms of this system of degree $i$ have Schmidt rank at least $\rho_i(2R_{\G_\HH}+2r_d,d)$. Indeed, suppose for $\la,\mu\in\QQ^r$ with say $\la\neq\0$
\[\la\cdot (\f_2(\z)-\f_2(\z'))+\mu\cdot (\f_2(\z)-\f_2(\vv'))=\sum_{i=1}^h U_i(\z,\z'\vv,) V_i(\z,\z'\vv,).\]
Substituting $\z'=\vv'=\0$ we have $\la\cdot \f_2(\z)=\sum_{i=1}^h U_i'(\z)\,V_i'(\z)$ and as not all terms in the sum can be zero and for each non-zero term $U_i'(\z)V_i'(\z)$ both forms have positive degrees we have that $h\geq h_\QQ(\f_2)$. The rank of the systems $\G_\HH^{(i)}$ are estimated the same way. The system \eqref{sys2} hence
\[\W_2 = O(N^{3(n-K-m)-2D_p-2D_{\G_\HH}}).\]
Putting these estimates together
\[\W\leq \sqrt{\W_1\W_2} = O(N^{2(n-K)-2D_\p-D_{\bar{\G}_\GG}-D_{\G_\HH}}).\]
This proves claim 2 and Proposition 3 follows.
\end{proof}

%%%%%%%
\bigskip

\section{The major arcs }

The treatment of the major arcs is fairly standard and differs little from the scenario with diagonal equations (see Hua \cite{Hua}). Recall that the major arcs are a union of boxes of the form
$\mathfrak{M}_{\mathbf{a},q}(C)$, where $q\leq (\log\,N)^C$ and $C$ is now a fixed constant chosen large enough so that lemma ~\ref{minor} holds with $c=1$. For a
fixed $\mathbf{a}$ and $q$,  the small size of the associated major arc means that the exponential sum\[
T_{\p}(\alpha)=\sum_{\x\in[N]^n} \La(\x) e(\p(\x)\cdot\alpha)\]
can be replaced by any approximation that has a sufficiently large logarithmic power gain in the error.

Upon the actual fixing of a $q\leq (\log\,N)^C$ and an
$\mathbf{a}\in U_q^n$, one has\footnote{There is some ambiguity in the case where $N$ is a prime power, however, there is no harm in assuming that this is not so due to the fact that the prime powers are sparse.}
\begin{eqnarray}\label{3.1}
T_{\p}(\alpha)&=&\sum_{\x\in[N]^n} \La(\x) e(\p(\x)\cdot\alpha) \nonumber\\
&=&\sum_{\mathbf{g}\in Z_q^n}\sum_{\x\in[N]^n}\mathbf{1}_{\x\equiv \mathbf{g} \, (q)}\La(\x)e(\mathbf{a}\cdot\p(\mathbf{g})/q)e(\p(\x)\cdot\tau) \\
&=&\sum_{\mathbf{g}\in Z_q^n}e(\p(\mathbf{g})\cdot\mathbf{a}/q)\int_{\mathbf{X}\in N\mathcal{J}} e(\p(\mathbf{X})\cdot\tau)d\psi_\mathbf{g}(\mathbf{X}), \nonumber
\end{eqnarray}
where the notations introduced here are $\tau_i=\alpha_i-a_i/q$, and
$\psi_\mathbf{g}(\mathbf{X})=\psi_{g_1}(X_1)...\psi_{g_n}(X_n)$ with
\[\psi_l(v)=\sum_{t\leq v,\,t\equiv l  \, (q)}\La(t),\] and $\mathcal{J}$
is the unit cube $[0,1]^n\subset\mathbb{R}^n$.

\begin{lemma}
For any given a constant $c$, the estimate
\eq\int_{\mathbf{X}\in N\mathcal{J}}
e(\p(\mathbf{X})\cdot\tau)d\psi_\mathbf{g}(\mathbf{X})=\mathbf{1}_{\mathbf{g}\in U_q^n}\phi(q)^{-n}\int_{\z\in
N\mathcal{J}}e(\p(\z)\cdot\tau)d\z+O((\log\,N)^{-c}N^{n})\ee
holds on each major arc $\mathfrak{M}_{a,q}(C)$ provided that $C$ is sufficiently large.
\end{lemma}

\begin{proof}
Define for a fixed $l$ the one dimensional signed measure
$d\nu_l=d\psi_l-d\omega_l$, where $d\omega_l$ is the Lebesgue
measure multipled by the reciprocal of the totient of $q$ if
$l\in U_q$, and is zero otherwise. For a continuous function
$f$ one then has\[ \int_{0}^N f(X)d\nu_l(X)=\sum_{x\in[N], \,
x\equiv l \,(q)}\La(x)f(x)-\phi(q)^{-1}\int_0^Nf(z)dz.\]

Now set
$d|\nu_l|=d\omega_l +d\psi_l$, so that \[ \int_{\mathbf{X}\in N\mathcal{J}} e(\p(\mathbf{X})\cdot\tau)d\psi_\mathbf{g}(\mathbf{X})=\int_{\mathbf{X}\in
N\mathcal{J}}
e(\p(\mathbf{X})\cdot\tau)\prod_{i=1}^n\left(d\nu_{g_i}(X_i)+d\omega_{g_i}(X_i)\right).\]
Expanding out the product in the last integral gives the form\[
\int_{\mathbf{X}\in N\mathcal{J}}e(\p(\mathbf{X})\cdot\tau)d\omega_\mathbf{g}(\mathbf{X}) +
\sum_{i=1}^{2^n-1}\int_{\mathbf{X}\in
N\mathcal{J}}e(\p(\mathbf{X})\cdot\tau)d\mu_{i,\mathbf{g}}(\mathbf{X}),\] where $d\mu_{i,\mathbf{g}}$ runs
over all of the corresponding product measures, barring the $d\omega_\mathbf{g}(\mathbf{X})$
term.

Consider\[ \int_{\mathbf{X}\in N\mathcal{J}}e(\p(\mathbf{X})\cdot\tau)d\mu_{i,\mathbf{g}}(\mathbf{X})\] for
some fixed $i$.  Assume without loss of generality that $d\mu_{i,\y}$
is of the form \[d\nu_{g_1}(X_1)d\sigma_\mathbf{g}(X_2,...,X_n),\] where
$d\sigma_\mathbf{g}$ may be signed in some variables and is
independent of $g_1$. The range of integration for the $X_1$ variable is a copy of
the continuous interval $[0,N]$, and is to be split into smaller disjoint intervals of
size $N^{1}(\log \, N)^{-c'}$. Here $c'$ is chosen to be between
$(c+C)$ and $2(c+C)$ such that $(\log \, N)^{c'}$ is an integer, say $B$. The
equality $[0,N]=\bigcup_{j=1}^B I_j$ follows. Also set
$\mathcal{J}'_j=I_j\times[0,N]^{n-1}$, which absorbs the factor of
$N$.

Now, for a fixed interval $I_j$,  select some $t\in I_j$. Then write
\begin{eqnarray}
\int_{\mathbf{X}\in \mathcal{J}'_j}e(\p(\mathbf{X})\cdot\tau)d\mu_{i,\mathbf{g}}
&=&\int_{\mathbf{X}\in\mathcal{J}_j'}e(\p(t,X_2,...,X_n)\cdot\tau)d\nu_{g_1}(X_1)d\sigma_\mathbf{g}(X_2,...,X_n) \nonumber\\
&&
+ \int_{\mathbf{X}\in \mathcal{B_j}}(e(\p(X_1,...,X_n)\cdot\tau)-e(\p(t,X_2,...,X_n)\cdot\tau))    \nonumber\\
&& \times d\nu_{g_1}(X_1)d\sigma_\mathbf{g}(X_2,...,X_n) \nonumber\\
&:=&E_1+E_2 \nonumber
\end{eqnarray}
The first error term satisfies \[ |E_1|\leq
\int_{X_2,...,X_n\in[0,N]}|\int_{I_j}d\nu_{g_1}(X_1)|\,\,d|\sigma_\mathbf{g}|(X_2,...,X_n)=O(N^{n}e^{-c_0\sqrt{\log\,N}})\]
for some positive constant $c_0$ by the Siegel-Walfisz theorem, as $q\leq (\log\,N)^C$. To bound
$E_2$, note that on $I_j$ the integrand is
\[O(\left|\p(X_1,...,X_n)-\p(t,X_2,...,X_n))\cdot\tau\right|)=O((\log\,N)^{C-c'}).\] In turn,\[ |E_2|=
O((\log\,N)^{C-c'}))\int_{\mathbf{X}\in\mathcal{J}_j'}d|\nu_{g_1}|(X_1)d|\sigma_\mathbf{g}|(X_2,...,X_n)=
O(N^{n}(\log\,N)^{C-2c'})).\]

There are $2^n-1$ error terms on each interval, so summing over the $(\log\,N)^{c'}$ intervals completes  the proof.
\end{proof}

The integral appearing in the last result has a quick reduction, namely\[
\int_{N\mathcal{J}}e(\p(\mathbf{X})\tau)d\mathbf{X}=\int_{N\mathcal{J}}e(\f(\mathbf{X})\tau)d\mathbf{X}+O(N^{n-1+\epsilon}),\]
recalling that $\f$ is the highest degree part of $\p$. Following along with the work of Birch,
\[
\int_{N\mathcal{J}}e(\f(\mathbf{X})\tau)d\mathbf{X}=N^n\int_{\zeta\in\mathcal{J}}e(\f(\zeta)\cdot N^d\tau)d\zeta,\]
is denoted by $N^n\mathcal{I}(\mathcal{J},N^d\tau)$ in \cite{Bi}.
This function is independent of $\mathbf{a}$ and $q$. Thus the integral over
any major arc yields the common integral \[ \int_{|\tau|\leq
(\log\,N)^C} \mathcal{I}(\mathcal{J},N^{d}\tau)e(-\s\cdot\tau )d\tau.\]
With $\mu=N^{-d}\s$, set \[
J(\mu;\Phi)=\int_{|\tau|\leq\Phi}\mathcal{I}(\mathcal{J},\tau)e(-\mu\cdot\tau
)d\tau,\] and\[ J(\mu)=\lim_{\Phi\rightarrow\infty}J(\mu;\Phi).\] The
following is Lemma 5.3 in \cite{Bi}.

\begin{lemma}
The function $J(\mu)$ is continuous and uniformly bounded in $\mu$. Moreover, \[
|J(\mu)-J(\mu,\Phi)|\lesssim \Phi^{-\frac{1}{2}}\]
holds uniformly in $\mu$.
\end{lemma}

By defining\[
W_{\mathbf{a},q}=\sum_{\mathbf{g}\in U_q^n}e(\p(\mathbf{g})\cdot\mathbf{a}/q),\]
one then has

\begin{lemma}
For any given $c>0$, the estimate\[
\int_{\mathfrak{M}_{\mathbf{a},q}(C)}T_{\p}(\alpha)e(-\s\cdot\alpha)d\alpha=N^{n-dr}\phi(q)^{-n}W_{\mathbf{a},q}e(-\s\cdot\mathbf{a}/q)J(\mu) + O(N^{n-dr}(\log\,N)^{-c}),\]
where $\mu=N^{-2}\s$, holds on each major arc $\mathfrak{M}_{\mathbf{a},q}(C)$.
\end{lemma}

The measure of the major arcs is easily at most $N^{-dr}(\log\,N)^{K}$ for some constant $K$. By defining
\medskip
\begin{equation}
\begin{array}{l}
\displaystyle B(\s,q)=\sum_{\mathbf{a}\in U_q^r}\phi(q)^{-n}W_{\mathbf{a},q}e(-\s\cdot\mathbf{a}/q)\\
\displaystyle \mathfrak{S}(\s,N)=\sum_{q\leq
(\log\,N)^C}B(\s,q), \nonumber\\
\end{array}
\end{equation}
it then follows that

\begin{lemma}
The estimate
\eq\label{asyweak}
\mathcal{M}_{\p,\s}(N)=\mathfrak{S}(\s,N)J(\mu)N^{n-dr}+O((\log\,N)^{-c}N^{n-dr})
\ee
holds for any chosen value of $c$.
\end{lemma}

%%%%%%%
\section{The singular series}

Following the outline  of Hua (\cite{Hua}, chapter VIII, \S2, Lemma 8.1), one can show that  $B(\s,q)$ is multiplicative as a function of $q$. This leads to consideration of the formal identity
\eq\label{product}\mathfrak{S}(\s):=\lim_{N\rightarrow \infty}\mathfrak{S}(\s,N)=\prod_{p<\infty}(1+\sum_{t=1}^\infty B(\s,p^t))).\ee

\begin{lemma}\label{bbound}

If $q=p^t$ is a prime power, then
\eq
B(\s,q)=O(q^{r-\B(\p)/((d-1)2^dr)+\epsilon})
\ee
holds uniformly in $\s$ as $t\rightarrow\infty$. The implied constants can be made independent of $p$.
\end{lemma}

\begin{proof}
It is shown here that \[
W_{\mathbf{a},q}=O(q^{n-\B(\p)/((d-1)2^dr)+\epsilon}),\]
uniformly for $\mathbf{a}\in U_q^n$, which clearly implies the result by the definition of $B(\s,q)$ and the fact that $q^n/\phi(q)^n\leq 2^n$ independent of $p$.

The  inclusion-exclusion principle is used to bound $W_{\mathbf{a},q}$ when $q=p^t$ when $t\leq d$. Let such a $t$ be fixed, and note that the characteristic function of $U_{p^t}$ decomposes as
\[
\mathbf{1}_{U_{p^t}}(x)=1-\sum_{h\in{Z_{p^{t-1}}}}\mathbf{1}_{x=hp}.\]
Applying this in the definition gives
\begin{eqnarray}
W_{\mathbf{a},q}&=&\sum_{\mathbf{g}\in U_q^n}e(\p(\mathbf{g})\cdot\mathbf{a}/q) \nonumber\\
&=&\sum_{\mathbf{g}\in Z_q^n}\prod_{i=1}^n(1-\sum_{h_i\in{Z_{p^{t-1}}}}\mathbf{1}_{g_i=h_ip})e(\p(\mathbf{g})\cdot\mathbf{a}/q) \nonumber\\
&=&\sum_{I\subseteq[n]}(-1)^{|I|}\sum_{\mathbf{h}\in Z_{p^{t-1}}^{|I|}}\sum_{\mathbf{g}\in Z_q^n}F_I(\mathbf{g};\mathbf{h})e(\p(\mathbf{g})\cdot\mathbf{a}/q),
\end{eqnarray}
where \[
F_I(\mathbf{g};\mathbf{h})=\prod_{i\in I}\mathbf{1}_{g_i=ph_i}\]
for $\mathbf{h}\in Z_p^{|I|}$. In other words, $F_I$ is the characteristic function of the set $H_{I,\mathbf{h}}=\{\mathbf{g}:g_i=ph_i \, \forall \, i\in I\}$.

The sets $I\subseteq[n]$ divided into two categories according to whether   $|I|\leq\B(\p)/(r+1)$ or not. If $I$ is a set fitting into the latter category, then the trivial estimate is\[
\left|\sum_{\mathbf{h}\in Z_{p^{t-1}}^{|I|}}\sum_{\mathbf{g}\in Z_q^n}F_I(\mathbf{g};\mathbf{h})e(\p(\mathbf{g})\cdot\mathbf{a}/q)\right|=p^{(t-1)|I|}(p^t)^{n-|I|}=(p^t)^{n-|I|/t}\leq q^{n-\B(\p)/(tr+t)}\leq q^{n-\B(\p)/((d-1)2^dr)}.\]

Now let  $I$ be a fixed subset of $[n]$ with $|I|\leq\B(\p)/(r+1)$.  For each $\mathbf{h}$ the restriction of $\p$ to the set $H_{I,\mathbf{h}}$ has Birch rank at least $\B(\p)-|I|(r+1)$ by corollary ~\ref{subspacerank}. By the work of Birch (\cite{Bi}, Lemma 5.4) it follows that
\begin{eqnarray}
\sum_{\mathbf{h}\in Z_p^{(t-1)|I|}}\left|\sum_{\mathbf{g}\in Z_q^n}F_I(\mathbf{g};\mathbf{h})e(\p(\mathbf{g})\cdot\mathbf{a}/q)\right|
&\<& q^{(t-1)|I|/t}q^{n-|I|-(\B(\p)-|I|(r+1))/((d-1)2^{d-1}r)+\epsilon} \nonumber\\
&\<&q^{n-\B(\p)/((d-1)2^{d}r)+\epsilon}.\nonumber
\end{eqnarray}
Summing over all $I$ yields the bound.

Now let $q=p^t$ for $t>d$. Going back to the definition gives
\begin{eqnarray}
W_{\mathbf{a},q}&=& \sum_{\mathbf{g}\in U_q^n}e(\p(\mathbf{g})\cdot\mathbf{a}/q) \nonumber\\
&=& \sum_{\mathbf{g}\in U_p^n}\sum_{\mathbf{h}\in Z_{p^{t-1}}^n}e(\p(\mathbf{g}+p\mathbf{h})\cdot\mathbf{a}/q)\nonumber.
\end{eqnarray}
The system of forms in the exponent can be expanded for each fixed $\mathbf{g}$ as \[\p(\mathbf{g}+p\mathbf{h})=p^d\p(\mathbf{h})+f_\mathbf{g}(\mathbf{h})\] for some polynomial $f_\mathbf{g}$ of degree at most $d-1$.  Then it follows that \eq\label{Wbound}
|W_{\mathbf{a},q}|\leq \sum_{\mathbf{g}\in U_p^n}\left|\sum_{\mathbf{h}\in Z_{p^{t-1}}^n}e(f_\mathbf{g}(\mathbf{h})+p^d\p(\mathbf{h}))\cdot\mathbf{a}/q)\right|.\ee
The inner sum is now bounded uniformly in $\mathbf{g}$ by an application of the exponential sum estimates in \cite{Bi} as follows.

Set $P=p^{t-1}$ and $q_1=p^{t-d}$. Then, for each $i=1,...,r$, \[
2|q'a_i -a_i'q_1|\leq P^{-(d-1) + (d-1)r\theta}\]
and\[
1\leq q'\leq P^{(d-1)r\theta}\]
cannot be satisfied if ${\theta}<1/(d-1)r$. Then, by Lemma 4.3 of \cite{Bi},\[
\sum_{\mathbf{h}\in Z_{p^{t-1}}^n}e((p^d\p(\mathbf{h}))\cdot\mathbf{a}/q+f(\h))=O(P^{n-\B(\p)/((d-1)2^dr)+\epsilon})\] for any polynomial $f(\h)$ of degree strictly less than $d$. In turn,\[
|W_{\mathbf{a},q}|\leq \sum_{\mathbf{g}\in U_p^n}O(P^{n-B(\p)/((d-1)2^dr)+\epsilon})=O(q^{n-B(\p)/((d-1)2^dr)+\epsilon}),\]
which is what is needed to complete the proof in this last and final case.
\end{proof}

Now define the local factor for a finite prime $p$ as
\eq
\mu_p=1+\sum_{t=1}^\infty B(\s,p^t)),
\ee
which is well defined as the series is absolutely convergent provided that $\B(\p)> (d-1)2^dr(r+1)$. The following result is again an straight forward extension of the results for a single form.
\begin{lemma} For each finite prime $p$, the local factor may be represented as
\eq\label{mu}
\mu_p=\lim_{t\rightarrow \infty}\frac{(p^t)^RM(p^t)}{\phi^n(p^t)},
\ee
where $M(p^t)$ represents the number of solutions to equation ~\ref{1} in the multiplicative group $U_{p^t}$.
 \end{lemma}

At our disposal now is the fact that the $\mu_p$ are positive, which then easily gives the following.

\begin{lemma}
If $\B(\p)>(d-1)2^dr(r+1)$ then the local factor for each finite prime satisfies the estimate \[
\mu_p=1+O(p^{-(1+\delta}))\]
for some positive $\delta$, and therefore the product in equation ~\ref{product} is absolutely convergent and thusly is in fact well defined.
\end{lemma}

The observation that \[
|\mathfrak{S}(\s,N)-\mathfrak{S}(\s)|=o(1)\]
gives the final form of the asymptotic
\[
\mathcal{M}_{\p,\s}(N)=\mathfrak{S}(\s)J(\mu)N^{n-dr}+O((\log\,N)^{-c}N^{n-dr}).
\]
This proves our main result, Theorem 1. The non-vanishing of local factors follows from the exact same lifting argument as when dealing with integer solutions \cite{Bi}.

\section{Further remarks.}

There are further possible refinements of the main result of this paper. It is expected, similarly to the case of integer solutions \cite{Sch}, that Theorem 1 holds assuming only the largeness of the rational Schmidt rank of the system. To prove this one needs to find a suitable analogue of Proposition 2 for the rational Schmidt rank instead of the Birch rank. The tower type bounds on the ranks are due to the regularization process expressed in Proposition 1. It is expected that exponential type lower bounds on the rank of the system are sufficient. It might be possible that there are "transfer principles" for higher degree forms, to allow a more direct transition to find solutions in the primes.

The methods of this paper may extend to other special sequences not just the primes. For example for a translation invariant system of forms one might modify our arguments to show the existence of solutions chosen from a set $A$ of upper positive density. Indeed, assuming that $A$ is sufficiently uniform, one expects that an analogue of Lemma 5 and hence our main estimate Proposition 3 holds, with the minor arcs replaced by $[0,1]^r$. Otherwise the balanced function of  $A$ should correlate with a polynomial phase function which naturally leads to a standard density increment argument. Note that the existence of solutions in the set $A$ already follows from known results, Szemer\'{e}di's theorem \cite{GT2} together with the main result of Birch \cite{Bi}, providing though weaker quantitative bounds on the density.
We do not pursue these matters here.

\end{document}